\documentclass{amsart}
\usepackage{amsfonts}
\usepackage{amssymb}
\usepackage{color}
\newtheorem{theorem}{Theorem}[section]
\newtheorem{teo}[theorem]{Theorem}
\newtheorem{lm}[theorem]{Lemma}

\newtheorem{pr}[theorem]{Proposition}
\newtheorem{ex}[theorem]{Example}

\DeclareMathOperator{\GL}{GL}
\DeclareMathOperator{\Dist}{Dist}
\DeclareMathOperator{\Hom}{Hom}
\DeclareMathOperator{\SAlg}{SAlg}
\DeclareMathOperator{\cont}{cont}
\DeclareMathOperator{\lead}{lead}
\DeclareMathOperator{\admissible}{admissible}
\DeclareMathOperator{\proj}{proj}
\begin{document}
\title[Combinatorial aspects of an odd linkage property for supergroups]{Combinatorial aspects of an odd linkage property for general linear supergroups}
\author{Franti\v sek Marko}
\address{The Pennsylvania State University, 76 University Drive, Hazleton, PA 18202, USA}
\email{fxm13@psu.edu}
\begin{abstract}
Let $G=\GL(m|n)$ be a general linear supergroup, and $G_{ev}$ be its even subsupergroup isomorphic to $\GL(m)\times \GL(n)$. 
In this paper, we use the explicit description of $G_{ev}$-primitive vectors in the costandard supermodule $\nabla(\lambda)$, the largest polynomial $G$-subsupermodule of the induced supermodule 
$H^0_G(\lambda)$, for $(m|n)$-hook partition $\lambda$, and properties of certain morphisms $\psi_k$ to derive results related to 
odd linkage for $G$ over a field $F$ of characteristic different from $2$.
\end{abstract}
\maketitle
\section*{Introduction}
Throughout the paper, except for its last section, we assume that the ground field $F$ is algebraically closed and of characteristic zero.

For the definition of the general linear supergroup $G=\GL(m|n)$, its distribution superalgebra $\Dist(G)$, and properties of its induced supermodules $H^0_G(\lambda)$, the reader is asked to consult \cite{brunkuj}.

The linkage principle for reductive algebraic groups, over a field of positive characteristic states, that a weight $\mu$, of the simple factor $L(\mu)$ appearing in the composition series of the induced module $H^0(\lambda)$, is obtained via a repeated application of the dot action of the corresponding affine Weyl group on the weight $\lambda$.
This linkage applied to the maximal even subgroup $G_{ev}\simeq \GL(m)\times \GL(n)$ of $G$ gives rise to the even linkage of weights. 
Over the ground field of characteristic zero, the even linkage is trivial because all $G_{ev}$-modules are semisimple.

However, for general linear supergroup $G$, there is another type of linkage - the odd linkage - that appears due to the presence of odd roots of $G$. The odd linkage is nontrivial even when the characteristic of the ground field is zero. For odd characteristic $p$, the linkage is a combination of even and odd linkages.

The primary focus of this paper is the odd linkage of weights of $G$ and its combinatorial aspects. 
If a weight $\mu$ is such that the simple supermodule $L_G(\mu)$ is a composition factor of $H^0_G(\lambda)$ and $\mu$ is odd-linked to $\lambda$, then it is of the form $\mu=\lambda_{I|J}$, where the pair $(I|J)$ of multi-indices $I=(i_1\cdots i_k)$ and $J=(j_1\cdots j_k)$ with $1\leq i_1, \ldots, i_k\leq m$ and $1\leq j_1, \ldots, j_k\leq n$ is admissible. Therefore from the very beginning, we concentrate our attention on weights of type $\lambda_{I|J}$ as above. The weights $\mu=\lambda_{I|J}$ are such that the simple module $L_G(\mu)$ is a potential composition factor of the induced $G$-supermodule $H^0_G(\lambda)$. A description of composition factors 
$L_G(\mu)$ of $H^0_G(\lambda)$ is referred to as a strong linkage. Aside from the description of blocks for $GL(m|n)$ given in \cite{mz}, a complete description of strong linkage by exhibiting actual composition factors $L_G(\mu)$ of $H^0_G(\lambda)$ is not known.

In the case of characteristic zero, we describe the strong linkage in terms of the surjectivity of a map $\psi_k$ - see Proposition \ref{comp}.
Using this result, we derive the strong linkage principle for robust weights using linear algebra and combinatorial methods in Proposition \ref{rob-link2}.
Versions of the above theorems are valid when the induced supermodule $H^0_G(\lambda)$ is replaced by its maximal polynomial subsupermodule $\nabla(\lambda)$.

For the polynomial supermodule $\nabla(\lambda)$, we can appply combinatorial techniques most efficiently.

The category of polynomial $GL(m|n)$-supermodules of degreee $r$ is equivalent to the category of supermodules over the Schur superalgebra $S_r=S(m|n,r)$.

Our main result, Theorem \ref{link}, states that if both $\lambda$ and $\lambda_{I|J}$ are dominant polynomial weights, and $L_G(\lambda_{I|J})$ is a composition factor of $\nabla(\lambda)$, then $\lambda$ and $\lambda_{I|J}$ are odd-linked through a sequence of polynomial weights. This amount to the linkage principle for the superalgebra $S_r$ in the case of characteristic zero. As far as we know, this is the first such result for Schur superalgebras.

The situation for odd characteristic is more complicated due to the presence of even linkage, that is nontrivial in this case. Partial results in this directions are obtained in the last section of the paper.

We work mostly with modules over the (even) subgroup $G_{ev}$ of $\GL(m|n)$, and one of our primary tools is the explicit description of $G_{ev}$-primitive vectors in 
$\nabla(\lambda)$ established in \cite{fm-p2}, using the terminology of marked tableaux that can be regarded as a realization of the concept of pictures in the sense of Zelevinsky (\cite{zel}) to the general linear supergroups setup.   

The outline of the paper is as follows. 

In Section 1, we show that the maps $\psi_k$ are $G_{ev}$-morphisms.  

In Section 2, we show that there are a plethora of explicit $G_{ev}$-primitive vectors $\pi_{I|J}$ in the domain of the maps $\psi_k$ and that all $G_{ev}$-primitive vectors in the codomain of $\psi_k$ are linear combinations of explicit vectors $\overline{\pi}_{I|J}$.

In Section 3, we compute images $\psi_k(\pi_{I|J})$ and establish preliminary results relating surjectivity of $\psi_k$ to the strong linkage for $G$. 

In Section 4, we demonstrate results related to the strong linkage principle when the weight $\lambda$ is $(I|J)$-robust. This case is easier to handle since there is a basis of $G_{ev}$-primitive vectors consisting of vectors $\overline{\pi}_{I|J}$. The main idea of the strong linkage is already visible in this case.

In Section 5, we consider only polynomial weights $\lambda$ and $\lambda_{I|J}$ corresponding to $(m|n)$-hook partititions, and apply combinatorial techniques (using tableaux, Clausen order, pictures) to prove statements related to the linkage principle for $G$ and the corresponding Schur superalgebras.

In Section 6, we formulate a few results connecting our previous investigations to the linkage principle for general linear supergroups over ground fields $F$ of characteristic $p>2$.

\section{Maps $\psi_k$}
 
From now on, assume that the characteristic of the ground field $F$ is zero. We consider the case of the ground field of odd characteristic at the end of the paper.

Write a generic $(m+n)\times(m+n)$-matrix $C=(c_{ij})$ in an $(m|n)$-block form 
\[C=\begin{pmatrix}C_{11}&C_{12}\\C_{21}&C_{22}\end{pmatrix}.\]
Let $A(m|n)$ be the superbialgebra freely generated by elements $c_{ij}$ for $1\leq i,j \leq m+n$ subject to the supercommutativity relation
\[c_{ij}c_{kl}=(-1)^{|c_{ij}||c_{kl}|}c_{kl}c_{ij},\] 
where the parity $|c_{ij}|=0$ for $1\leq i,j \leq m$ or $m+1\leq i,j\leq m+n$ and $|c_{ij}|=1$ otherwise.
The comultiplication $\Delta$, and counit $\epsilon$ of $A(m|n)$ are defined as 
\[\Delta(c_{ij})=\sum_{k=1}^{m+n} c_{ik}\otimes c_{kj} \text { and } \epsilon(c_{ij})=\delta_{ij} \text{ for } 
1\leq i,j\leq m+n. \] 

The coordinate superalgebra $F[G]$ of the supergroup $G=\GL(m|n)$ is a localization of $A(m|n)$ by the element $det(C_{11})det(C_{22})$.
The general linear supergroup $G=\GL(m|n)$ is the representable functor from supercomutative superalgebras $\SAlg_F$ to groups defined as
\[\GL(m|n)(A)= \Hom_{\SAlg_F}(F[G], A) \text{ for } A\in \SAlg_F.\]
Denote by $A_{ev}(m|n)$ the subsuperalgebra of $A(m|n)$ spanned by the elements $c_{ij}$ for $1\leq i,j \leq m$ or $m+1\leq i,j\leq m+n$, and by 
$K(m|n)$ the localization $A(m|n)(A_{ev}(m|n)\setminus 0)^{-1}$. 

The maximal even subsupergroup $G_{ev}$ of $G$ satisfies $G_{ev}\simeq \GL(m)\times \GL(n)$.
The supergroups $G$ and $G_{ev}$ have the same standard maximal torus $T\simeq (K^*)^{m+n}$ consisting of diagonal matrices. The weights of $G$ and $G_{ev}$ are identical, and will be denoted by 
\[\lambda=(\lambda_1, \ldots, \lambda_m|\lambda_{m+1}, \ldots, \lambda_{m+n})=(\lambda_1^+, \ldots, \lambda^+_m|\lambda^-_1, \ldots, \lambda^-_n).\] 
We identify it with a pair of weights 
$\lambda^+=(\lambda^+_1, \ldots, \lambda^+_m)$ of $\GL(m)$ and $\lambda^-=(\lambda^-_1, \ldots, \lambda^-_n)$ of $\GL(n)$.
Denote by $\delta_s$ the weight of $G$ such that $(\delta_s)_s=1$ and all other components $(\delta_s)_t$ for $t\neq s$ vanish. Then the weight of the element $c_{ij}$ is $\delta_j$.

We work inside the induced supermodule $H^0_G(\lambda)$ considered using  
the superspace isomorphism $\tilde{\phi}$ given in Lemma 5.1 of \cite{zubal} as $H^0_{G_{ev}}(\lambda)\otimes S(C_{12})\to H^0_G(\lambda)$, where $S(C_{12})$ is the supersymmetric superalgebra of the superspace $C_{12}$. The map $\tilde{\phi}$ is a restriction of the multiplicative morphism $\phi:F[G]\to F[G]$ given on generators as
follows: 
\begin{equation*}
C_{11}\mapsto C_{11}, C_{21}\mapsto C_{21}, C_{12}\mapsto
C_{11}^{-1}C_{12}, C_{22}\mapsto C_{22}-C_{21}C_{11}^{-1}C_{12}.
\end{equation*}
It follows from p.985 of \cite{fm-p1} that with appropriately defined $G_{ev}$-module structure on $C_{12}$, the map $\tilde{\phi}$ is an isomorphism of $G_{ev}$-supermodules.

Let $D$ be the determinant of $C_{11}$ of weight $\delta_1+\ldots+\delta_m$, and $A=(A_{ij})$ be the adjoint matrix of $C_{11}$. Then 
\begin{equation*}
y_{ij}=\phi(c_{ij})=\frac{A_{i1}c_{1j}+A_{i2}c_{2j}+\ldots +A_{im}c_{mj}}{D}
\end{equation*}
for $1\leq i\leq m$ and $m+1\leq j \leq m+n$ has weight $-\delta_i+\delta_{j}$. The induced supermodule $H^0_G(\lambda)$ is identified with $\oplus_{k=0}^{mn} V\otimes \wedge^k Y$, where $V=H^0_{G_{ev}}(\lambda)$ is the even-induced supermodule and $Y$ is the span of the elements $y_{ij}$ for $1\leq i\leq m,m+1\leq j\leq m+n$.

The structure of induced modules over general linear groups is described using bideterminants. Since we consider $H^0_{G_{ev}}(\lambda)$ embedded inside $H^0_G(\lambda)$ using the map $\phi$,
we need to adjust the notation for bideterminants to accommodate the effect of the map $\phi$.

For $1\leq i_1, \ldots, i_s \leq m$ denote by $D^+(i_1, \ldots, i_s)$
the determinant 
\begin{equation*}
\begin{array}{|ccc|}
c_{1,i_1} & \ldots & c_{1,i_s} \\ 
c_{2,i_1} & \ldots & c_{2,i_s} \\ 
\ldots & \ldots & \ldots \\ 
c_{s,i_1} & \ldots & c_{s,i_s}%
\end{array}%
\end{equation*}
of weight $\delta_{i_1}+ \ldots + \delta_{i_s}$.
Clearly, if some of the numbers $i_1, \ldots, i_s$ coincide, then $%
D^+(i_1,\ldots, i_s)=0$.

For $m+1\leq j_1, \ldots, j_s \leq m+n$ denote by $D^-(j_1, \ldots, j_s)
$ the determinant 
\begin{equation*}
\begin{array}{|ccc|}
\phi(c_{m+1,j_1}) & \ldots & \phi(c_{m+1,j_s}) \\ 
\phi(c_{m+2,j_1}) & \ldots & \phi(c_{m+2,j_s}) \\ 
\ldots & \ldots & \ldots \\ 
\phi(c_{m+s,j_1}) & \ldots & \phi(c_{m+s,j_s})%
\end{array}%
\end{equation*}
of weight $\delta_{j_1}+\ldots + \delta_{j_s}$.
Clearly, if some of the numbers $j_1, \ldots, j_s$ coincide, then $%
D^-(j_1,\ldots, j_s)=0$.

The images of the highest vector $v_+$ of $H^0_{\GL(m)}(\lambda^+)$ and the highest vector $v_-$ of $H^0_{\GL(n)}(\lambda^-)$ under the map $\phi$ are identified 
as the following elements of $H^0_{G_{ev}}(\lambda)$: 
\[v^+=\prod_{a=1}^m D^+(1,\ldots, a)^{\lambda^+_a-\lambda^+_{a+1}}, \qquad
v^-=\prod_{b=1}^n D^-(m+1, \ldots, m+b)^{\lambda^-_b - \lambda^-_{b+1}},\]
where $\lambda^+_{m+1}=0=\lambda^-_{n+1}$. The product $v=v^+v^-$ is the highest vector of $H^0_G(\lambda)$ of weight $\lambda$.

The superderivation $_{ij}D$ of parity $|_{ij}D|=|c_{ij}|$ is given by $(c_{kl}) _{ij}D=\delta_{li}c_{kj}$, where $\delta_{li}$ stands for the Kronecker delta. It satisfies the property 
\[(ab)_{ij}D=(-1)^{|b||_{ij}D|}(a)_{ij}Db+a(b)_{ij}D\] for $a,b\in A(m|n)$. The action of $_{ij}D$ extends to $F[G]$ using 
the quotient rule
\[(\frac{a}{b})_{ij}D=\frac{(a)_{ij}Db-a(b)_{ij}D}{b^2}\] 
for $a,b\in A(m|n)$ and $b$ even.

The action of $_{ij}D$ on elements of $H^0_G(\lambda)$ was computed in Section 2 of \cite{fm-p1}. 

For $k=0, \ldots, mn$ denote by $T_k$ the supermodule $V\otimes Y^{\otimes k}$ and by $F_k$ the supermodule $V\otimes \wedge^k Y$ of $H^0_G(\lambda)$. The supermodule $F_k$ is called the $k$-\emph{floor} of $H^0_G(\lambda)$. It is essential to note that both $T_k$ and $F_k$ are $G_{ev}$-modules.

Let us denote $\ell(\mu^+)=\sum_{i=1}^m \mu^+_i$, $\ell(\mu^-)=\sum_{j=1}^n \mu^-_j$ and $\ell(\mu)=\ell(\mu^+|\mu^-)=\ell(\mu^+)+\ell(\mu^-)$. If $M$ is an indecomposable $G$-supermodule, the value of $\ell(\mu)$ for all non-zero weight spaces $M_{\mu}$ remains constant. If $M$ is an indecomposable $G_{ev}$-module, then 
$\ell(\mu^+)$ and $\ell(\mu^-)$ on all nonzero weight spaces $M_{\mu}$ remain constant. Therefore, the study the $G_{ev}$-structure of $H^0_G(\lambda)$ leads naturally to the grading by $H^0_G(\lambda)$ by the floors $F_k$ as above.   

In earlier papers \cite{fm1, gm}, we have considered $G_{ev}$-morphisms $\phi_k: V\otimes \wedge^k Y \to V\otimes \wedge^k Y$ defined by 
$\phi_k(v\otimes (y_{i_1j_1}\wedge \cdots \wedge y_{i_kj_k}))=(v)_{i_1j_1}D\cdots _{i_kj_k}D$ that proved useful when investigating the $G_{ev}$-module structure of $F_k$.

The focus of this paper are maps
$\psi_k:T_k \to F_k$ for $k=0, \ldots, mn$ defined as 
\[\begin{aligned}&\psi_k(v\otimes (y_{i_1j_1}\otimes \cdots \otimes y_{i_kj_k}))=(v\wedge y_{i_1j_1}\wedge \cdots \wedge y_{i_{k-1}j_{k-1}})_{i_kj_k}D=\\
&(-1)^{k-1}(v)_{i_kj_k}D\wedge y_{i_1j_1}\wedge \cdots \wedge y_{i_{k-1}j_{k-1}}+\\
&\sum_{l=1}^{k-1} (-1)^{k-l-1} v\wedge y_{i_1j_1}\wedge \cdots \wedge (y_{i_lj_k}\wedge y_{i_kj_l})\wedge \cdots \wedge y_{i_{k-1}j_{k-1}},
\end{aligned}\]
where the second equality follows from Lemma 2.1 of \cite{fm-p1}.

We abuse the notation, and instead of expressions like $v\otimes(y_{i_1j_1}\wedge \cdots \wedge y_{i_{k-1}j_{k-1}})$, 
we write $v\wedge y_{i_1j_1}\wedge \cdots \wedge y_{i_{k-1}j_{k-1}}$ and so on. Also, we identify expressions like 
$(v)_{i_kj_k}D\wedge y_{i_1j_1}\wedge \cdots \wedge y_{i_{k-1}j_{k-1}}$
belonging to $(V\otimes Y)\wedge Y \cdots \wedge Y$ with the corresponding element of $V\otimes (Y \wedge \cdots \wedge Y)$ and so on. This should not lead to any confusion, but it simplifies the exposition.

Of course $\phi_0=\psi_0$ and $\phi_1=\psi_1$.

Let $G$ be an arbitrary supergroup, and $A=F[G]$ be its coordinate superalgebra. Let $\Dist(G)$ be the superalgebra
of distributions of $G$ and $Lie(G)\subseteq \Dist(G)$ be the Lie superalgebra of $G$.
One can define (left) actions of $\Dist(G)$ on $A$ by 
\[\phi\cdot a=\sum a_1\phi(a_2)\] and by \[\phi\star a=\sum (-1)^{|\phi||a_1|} a_1\phi(a_2),\] respectively, where $\Delta(a)=\sum a_1\otimes a_2$. 

If $\phi\in Lie(G)$, then $\phi$ acts on $A$ as a right superderivation concerning the action $\cdot$, and as a left superderivation concerning the action $\star$. 
The relationship between both actions is $\phi\cdot a=(-1)^{|\phi|(|a|+1)}\phi\star a$ for $a\in A$.
The $\Dist(G)$-subsupermodules of $A$ for the $\star$ action are the same as those for to the $.$ action, and the morphisms of supermodules for the $\star$ action correspond to morphisms of supermodules for the $.$ action.

The proof of the next Lemma uses the relationship of superderivation $_{ij}D$ to action of the distribution algebra $\Dist(G)$.

\begin{lm}\label{evmorph}
Every map $\psi_k$, as above, is a $G_{ev}$-morphism.
\end{lm}
\begin{proof}
Since the action of a superderivation $_{ij}D$ on $A(m|n)$ corresponds to the $\cdot$ action of $e_{ji}$ on $A(m|n)$ - see \cite{lsz}, we obtain 
\[\begin{aligned}&\psi_k(v \otimes y_{i_1 j_1}\otimes \cdots \otimes y_{i_{k-1} j_{k-1}} \otimes y_{i_k j_k})=(v \wedge y_{i_1 j_1}\wedge \cdots \wedge y_{i_{k-1} j_{k-1}})_{i_k j_k}D\\
&=e_{j_k i_k}\cdot (v \wedge y_{i_1 j_1}\wedge \cdots \wedge y_{i_{k-1} j_{k-1}}),
\end{aligned}\]

Using the comments preceding Lemma \ref{evmorph} and the corresponding modification of Lemma 3 of \cite{mz}, we derive that 
for every $g\in G_{ev}$, there is 
\[\begin{aligned} &g(\psi_k(v\otimes y_{i_1 j_1}\otimes \cdots \otimes y_{i_{k-1} j_{k-1}} \otimes y_{i_k j_k})) 
= g(e_{j_k i_k}\cdot (v \wedge y_{i_1 j_1}\wedge \cdots \wedge y_{i_{k-1} j_{k-1}}))\\
&=Ad(g)(e_{j_k i_k})\cdot (g(v \wedge y_{i_1 j_1}\wedge \cdots \wedge y_{i_{k-1} j_{k-1}})).
\end{aligned}\]
Since the adjoint action of $h\in G_{ev}$ on $e_{ji}$, where $1\leq i\leq m<j\leq m+n$, is given as 
\[Ad(h)e_{ji}=he_{ji}h^{-1}=\sum_{l=m+1}^{m+n} \sum_{k=1}^m  h_{lj} e_{lk} (h^{-1})_{ik}\] and its action on $y_{ij}$ is given as
\[hy_{ij}=\sum_{k=1}^m\sum_{l=m+1}^{m+n} (h^{-1})_{ik}y_{kl}h_{lj},\]
we conclude that the map $y_{ij}\mapsto e_{ji}$ induces an isomorphism of $G_{ev}$-supermodules $Y=\sum_{1\leq i\leq m < j\leq m+n}F y_{ij}$ and 
$\sum_{1\leq i\leq m < j\leq m+n}F e_{ji}$ and the claim follows.
\end{proof}

\section{Even-primitive vectors $\pi_{I|J}$}

\subsection{Notation}

We start by recalling the definition of a $G_{ev}$-primitive vector $w$.
Let $U^-_{ev}$ be the unipotent subgroup of $G_{ev}$ corresponding to lower triangular matrices from $G_{ev}$ with all diagonal entries equal to 1. A vector $w$ of a module $M$ is called $G_{ev}$-primitive (or even-primitive) vector of $M$
if every element of $U^-_{ev}$ annihilates $w$.

Next, we review the definition of $\pi_{I|J}$, $v_{I|J}$, $\rho_{I|J}$ from \cite{fm-p1}.

Assume that $(I|J)=(i_1\ldots i_k|j_1\ldots j_k)$ is a multi-index such that $1\leq i_1, \ldots, i_k \leq m$ and $1\leq j_1, \ldots, j_k\leq n$. We define the length of $(I|J)$ as $k$.
Define the content $\cont(I)=(x_1^+, \ldots, x^+_m)$, where $x^+_s$ is the number of occurrences of the symbol $s$ in $i_1\ldots i_k$, the content $\cont(J)= (x^-_{1}, \ldots, x^-_{n})$, where $x^-_t$ is the number of occurrences of the symbol $t$ in $j_1\ldots j_k$. Finally, define $\cont(I|J)=(-\cont(I)|\cont(J))$.

Further, denote
\begin{equation}\label{laIJ}
\lambda_{I|J}=\lambda-\sum_{s=1}^k \delta_{i_s}+\sum_{s=1}^k \delta_{m+j_s},
\end{equation}
In particular, if $I=\{i\}$ and $J=\{j\}$, then 
\begin{equation}\label{laij}
\lambda_{ij}=\lambda_{I|J}=(\lambda^+_1, \ldots, \lambda^+_i-1, \ldots, \lambda^+_m|\lambda^-_1, \ldots, \lambda^-_j+1, \ldots, \lambda^-_n).
\end{equation}

It is important to note that if the weight space of $H^0_G(\lambda)$ corresponding to the weight $\lambda_{I|J}$ is nonzero, then $\lambda_{I|J}$ is a weight of the $k$th floor $F_k$, $\ell(\lambda_{I|J}^+)=\ell(\lambda^+)-k$, 
and $\ell(\lambda_{I|J}^-)=\ell(\lambda^-)+k$.

The main reason to consider the weights $\lambda_{I|J}$ is that they are the only weights of $H^0_G(\lambda)$ that can be odd-linked to $\lambda$. For the definition of odd linkage, see Section \ref{sec4}.

For $1\leq i\leq m$ and $1\leq j \leq n$ we set 
\[\rho_{ij}=\sum_{r=i}^m D^+(1, \ldots, i-1,r)\sum_{s=1}^{j} (-1)^{s+j} D^-(m+1, \ldots, \widehat{m+s},\ldots, m+j)y_{r,m+s}.\]
The weight of the element $\rho_{ij}$ is $\sum_{1\leq t<i} \delta_t + \sum_{1\leq t\leq j} \delta_{m+t}$.

For each $(I|J)=(i_1\cdots i_k|j_1\cdots j_k)$ such that $1\leq i_1, \ldots, i_k \leq m$ and $1\leq j_1, \ldots, j_k\leq n$
denote the element $y_{i_1,m+j_1}\otimes \cdots \otimes y_{i_k,m+j_k}$ by $y_{I|J}$, the element $y_{i_1,m+j_1}\wedge \cdots \wedge y_{i_k,m+j_k}$ by $\overline{y}_{I|J}$
and $\otimes_{s=1}^k \rho_{i_sj_s}$ by $\rho_{I|J}$.
Abusing the notation slightly, we consider $\rho_{I|J}$ as an element of $A(m|n)\otimes Y^{\otimes  k}$ by changing the order of terms in $\rho_{i_1j_1}\otimes \ldots \otimes \rho_{i_kj_k}$.
The weight of the element $y_{I|J}$ is $-\sum_{s=1}^k \delta_{i_s} +\sum_{s=1}^k \delta_{m+j_s}$, and the weight of $\rho_{I|J}$ is 
\[\sum_{s=1}^k (\sum_{1\leq t<i_s}\delta_t+ \sum_{1\leq t\leq j_s} \delta_{m+t}).\]

For such $(I|J)$ denote the elements
\[v_{I|J}=\frac{v}{\prod_{s=1}^k D^+(1, \ldots, i_s)\prod_{s=1}^k D^-(m+1,\ldots, m+j_s-1)},\] where 
we set $D^-(m+1,\ldots, m+j_s-1)=1$ for $j_s=1$.
Recall that 
\[v=\prod_{a=1}^m D^+(1,\ldots, a)^{\lambda^+_a-\lambda^+_{a+1}}\prod_{b=1}^n D^-(m+1, \ldots, m+b)^{\lambda^-_b - \lambda^-_{b+1}},\]
where $\lambda^+_{m+1}=0=\lambda^-_{n+1}$.
A slight abuse of notation allows us to consider $v_{I|J}$ as an element in $Q_k=K(m|n)\otimes Y^{\otimes k}$.
The weight of $v_{I|J}$ is 
\[\lambda-\sum_{s=1}^k(\sum_{1\leq t\leq i_s} \delta_t+\sum_{1\leq t<j_s} \delta_{m+t}).\]

As in the proof of Lemma 4.1 of \cite{fm-p1}, we can verify that $v_{I|J}$ and $\rho_{I|J}$ are annihilated by 
all divided power $_{ij}D^{(r)}$ of even superderivation $_{ij}D$ , which implies that 
the element 
\begin{equation}\label{piIJ}
\pi_{I|J}=v_{I|J}\rho_{I|J}
\end{equation}
is a $G_{ev}$-primitive vector in $Q_k$ of weight $\lambda_{I|J}$.

Denote by $\overline{\rho}_{I|J}$ the images of $\rho_{I|J}$ under the natural map $Q_k=K(m|n)\otimes Y^{\otimes k} \to \overline{Q}_k=K(m|n) \otimes \wedge^k Y$, and 
$\overline{\pi}_{I|J} = v_{I|J}\overline{\rho}_{I|J}$. 

Denote by $\lhd$ the dominance order on weights of $G$ given by $\mu \lhd \lambda$ if and only if
$\lambda-\mu$ is a sum of the simple roots of $G$.

If the weight of $u$ is $\zeta$, and $\cont(K|L)=\nu$, then the weight of $u\otimes y_{K|L}$
is $\zeta+\nu\unlhd \nu$. Since $T_k=H^0_{G_{ev}}\otimes Y^{\otimes k}$, 
every $w \in T_k$ of weight $\mu$ can be written in the form $w=\sum_{\kappa\unlhd \lambda} w_{\kappa}$,
where 
\[w_{\kappa}=\sum\limits_{\substack{(K|L)\\ \cont(K|L)=\mu-\kappa}}  w_{K|L}\otimes y_{K|L},\]
each $w_{K|L}$ is a vector from $H^0_{G_{ev}}(\lambda)$ of weight $\kappa$, and $(K|L)$ is of length $k$.

We say that a weight $\kappa \unlhd \lambda$ is a leading weight in $w$ if $w_\kappa\neq 0$ and $w_{\mu}=0$ for every $\mu$ such that $\kappa \lhd \mu \unlhd \lambda$.
Note that there can be more than one leading weight for $w$.

We call $(I|J)$ admissible if $\lambda_{I|J}$ is dominant, $i_1\leq \ldots \leq i_k$, and if indices $t_1<t_2$ are such that $i_{t_1}=i_{t_2}$, then $j_{t_1}<j_{t_2}$.

For $w\in F_k$, we have $w=\sum_{\kappa\unlhd \lambda} w_{\kappa}$, where 
\[w_{\kappa}=\sum\limits_{\substack{(K|L) \, \admissible \\ \cont(K|L)=\mu-\kappa}}  w_{K|L}\otimes \overline{y}_{K|L},\]
$w_{K|L}\in H^0_{G_{ev}}(\lambda)$ is of weight $\kappa$ and $(K|L)$ is of length $k$.  The definition of a leading weight is analogous to the one given above.

We apply similar descriptions to elements from $Q_k$ and $\overline{Q}_k$, in particular to vectors $\pi_{I|J}$ and $\overline{\pi}_{I|J}$.
Then $\lambda$ is the unique leading weight of both $\pi_{I|J}$ and $\overline{\pi}_{I|J}$, while $(\pi_{I|J})_{\lambda}=vy_{I|J}$ and $(\overline{\pi}_{I|J})_{\lambda}=v\overline{y}_{I|J}$.

\subsection{Even-primitive vectors of weights $\lambda_{I|J}$ in $H^0_G(\lambda)$}

Let $(I|J)$ be an admissible multi-index. Following \cite{fm-p1}, we say that the weight $\lambda$ is {\it $(I|J)$-robust} provided 
the symbol $i_s<m$ appears at most $\lambda^+_{i_s}-\lambda^+_{i_s+1}$ times in $I$,
symbol $m$ appears at most $\lambda^+_m$ times in $I$, and 
symbol $j_t>1$ appears at most $\lambda^-_{j_t-1}-\lambda^-_{j_t}$ times in $J$.

In this paper, we consider robust weights first since the even-primitive vectors of weight $\lambda_{I|J}$ such that $\lambda$ is $(I|J)$-robust have an easy description; many results have more transparent formulations and proofs for robust weights. Later we handle the general case.

Let us first describe a basis of even-primitive vectors in $F_k$ of weight $\mu$ under the assumption that the weight $\lambda$ is $(I|J)$-robust and $\mu=\lambda+\cont(I|J)$.

\begin{pr}\label{basis}
Assume the weight $\lambda$ is $(I|J)$-robust. Then the vectors $\pi_{K|L}$, where $\cont(K|L)=\cont(I|J)$, form a basis of even-primitive vectors of weight $\lambda_{I|J}$ in $T_k$. Additionally, the vectors $\overline{\pi}_{K|L}$, where $\cont(K|L)=\cont(I|J)$ and $(K|L)$ is admissible, form a basis of even-primitive vectors of weight $\lambda_{I|J}$ in $F_k$. 
\end{pr}
\begin{proof}
Since $\lambda$ is $(I|J)$-robust and $\cont(K|L)=\cont(I|J)$, $\lambda$ is also $(K|L)$-robust, $\pi_{K|L}\in T_k$ and $\overline{\pi}_{K|L}\in F_k$.
Lemma 4.1 of \cite{fm-p1} implies that vectors $\pi_{K|L}$ are even-primitive vectors in $T_k$, and $\overline{\pi}_{K|L}$ are even-primitive vectors in $F_k$.
It is clear that vectors $\pi_{K|L}$, where $\cont(K|L)=\cont(I|J)$, are linearly independent vectors of weight $\lambda_{I|J}$.  
By Lemmas 4.2 and 4.3 of \cite{fm-p1}, the vectors $\overline{\pi}_{K|L}$, where $\cont(K|L)=\cont(I|J)$ and $(K|L)$ is admissible, are linearly independent of weight $\lambda_{I|J}$. 

Let $w\neq 0$ be an even-primitive vector of weight $\mu=\lambda_{I|J}$ from $F_k$ and write $w=\sum_{\kappa\unlhd \lambda} w_{\kappa}$ as above. We show that $\lambda$ is the leading weight of $w$. Assume this not the case, and that $\nu\neq \lambda$ is a leading weight of $w$. 

Let $_{ij}D$ be an even superderivation. Then $(w)_{ij}D=0$ implies $(w_{\nu})_{ij}D=0$. Since 
$w_{\nu}=\sum\limits_{\substack{(K|L) \, \admissible \\ \cont(K|L)=\mu-\nu}}  w_{K|L}\otimes \overline{y}_{K|L}$
and $\nu$ is a leading weight of $w$, from the action of $_{ij}D$ on elements of $H^0_G(\lambda)$ and $\overline{y}_{K|L}$, we conclude that 
$(w_{K|L})_{ij}D=0$ for each admissible $(K|L)$ with $\cont(K|L)=\mu-\nu$. Since this is true for every even superderivation $_{ij}D$, we obtain that each such 
$w_{K|L}$ is an even-primitive vector of $H^0_{G_{ev}}(\lambda)$ of weight $\nu$. Since the characteristic of the field $F$ is zero, 
the $G_{ev}$-module $H^0_{G_{ev}}(\lambda)$ is irreducible, and its only nonzero primitive vectors are of the weight $\lambda$. Therefore, $w_{K|L}=0$ for 
each admissible $(K|L)$ with $\cont(K|L)=\mu-\nu$, which implies $w_\nu=0$ and contradicts the assumption that $\nu\neq \lambda$ is a leading weight of $w$.

Finally, let $w$ be an even-primitive vector of weight $\mu$ in $F_k$ and write 
$w_{\lambda}= \sum_{(K|L) \admissible } w_{K|L} \overline{y}_{K|L}$. Then $w-\sum_{(K|L) \admissible} w_{K|L} \overline{\pi}_{K|L}$ is an even-primitive vector in $F_k$, which does not have $\lambda$ as its leading weight. By the above argument, we obtain $w=\sum_{(K|L) \admissible} w_{K|L} \overline{\pi}_{K|L}$ showing that the vectors $\overline{\pi}_{K|L}$ for $(K|L)$ admissible and weight $\mu$ form a basis of all even-primitive vectors of weight $\mu$ in $F_k$.

The proof that every even-primitive vector in $T_k$ of weight $\lambda_{I|J}$ is a linear combination of vectors $\pi_{K|L}$, where $\cont(K|L)=\cont(I|J)$, is analogous.
\end{proof}

Note that if all entries in $I$ are distinct, all entries in $J$ are distinct, and $\lambda$ is $(I|J)$-robust, then the dimension of even-primitive vectors of weight $\lambda_{I|J}$ is $k!$.  

In the general case, even-primitive vectors of weight $\lambda_{I|J}$ in $H^0_G(\lambda)$ are 
certain linear combinations of elements $\overline{\pi}_{K|L}$ with $\cont(K|L)=\cont(I|J)$. 

\begin{pr}\label{genbasis} 
Every even-primitive vector $w$ in $T_k$ of weight $\lambda_{I|J}$ is a linear combination of vectors
$\pi_{K|L}$, where each $\cont(K|L)=\cont(I|J)$. Additionally, every even-primitive vector $w$ in $F_k$ of weight $\lambda_{I|J}$ is a linear combination of vectors
$\overline{\pi}_{K|L}$, where each $(K|L)$ is admissible and $\cont(K|L)=\cont(I|J)$.
\end{pr}
\begin{proof}
Denote by $M$ the $G_{ev}$-submodule of $\overline{Q}_k$ generated by $F_k$  and vectors $\overline{\pi}_{K|L}$, where $\cont(K|L)=\cont(I|J)$.
Assume that $w\in M$ is an even-primitive vector of weight $\lambda_{I|J}$ and write $w=\sum_{\kappa\unlhd \lambda} w_{\kappa}$ as before. The same argument as in Proposition \ref{basis} 
gives that $\lambda$ is the leading weight of $w$. 
If the leading term of $w$ is $w_{\lambda}= \sum_{(K|L) \admissible } c_{K|L} \overline{y}_{K|L}$, then $w-\sum_{(K|L) \admissible} c_{K|L} \overline{\pi}_{K|L}$ is an even-primitive vector in $M$ that does not have $\lambda$ as its leading weight. Therefore we conclude that $w=\sum_{(K|L) \admissible} c_{K|L} \overline{\pi}_{K|L}$. 

The proof for even-primitive vectors in $T_k$ is analogous.
\end{proof}

Theorem 7.1 of \cite{fm-p2} gives an explicit description of the basis of all even-primitive vectors in $\nabla(\lambda)$ for a hook partition $\lambda$, which will be used later. 

\section{Image of vectors $\pi_{I|J}$ under the map $\psi_k$}

The primary purpose of this section is to establish criteria for the strong linkage of weights. 

A weight $\mu$ is called \emph{strongly linked} to a weight $\lambda$ if the simple $G$-supermodule $L_G(\mu)$ appears as a composition factor in the induced $G$-supermodule $H^0_G(\lambda)$.

\subsection{Image $\psi_k(\pi_{I|J})$}
It is important to note that the vectors $\pi_{I|J}$ do not necessarily belong to $V\otimes Y^{\otimes k}$. 
Therefore, we extend the previously defined map $\psi_k$ naturally to a map from $Q_k$ to $\overline{Q}_k$. 
By abuse of notation, we denote this new map by the same symbol $\psi_k:Q_k\to \overline{Q}_k$.

Let $E$ be a linear combination of elements $\overline{\pi}_{K|L}$ for admissible $(K|L)$ of the same content as $(I|J)$.
The vectors $\overline{\pi}_{K|L}$ are linearly independent by Lemma 4.3 of \cite{fm-p1}.
We define $\lead(E)$ to be the leading term of $E$, which is the linear combination of all terms in $E$ that are scalar multiples of expressions of type $v\overline{y}_{M|N}$.

Since the leading terms $v\overline{y}_{K|L}$ of elements $\overline{\pi}_{K|L}$ are by itself linearly independent, to determine the coefficients of $\overline{\pi}_{K|L}$ in $\psi_k(\pi_{I|J})$, it is enough to determine the coefficients of their leading terms $v\overline{y}_{M|N}$. 

We write $\rho_{ij}$ as 
\[\begin{aligned} & D^+(1, \ldots, i-1,i) D^-(m+1, \ldots, \widehat{m+j})y_{i,m+j} \\
&+\sum_{r=i+1}^m D^+(1, \ldots, i-1,r) D^-(m+1, \ldots, \widehat{m+j})y_{r,m+j}\\
&+\sum_{s=1}^{j-1} D^+(1, \ldots, i-1,i) (-1)^{s+j} D^-(m+1, \ldots, \widehat{m+s},\ldots, m+j)y_{i,m+s}\\
&+F_{ij},\end{aligned}\]
where 
\[F_{ij}=\sum_{r=i+1}^m D^+(1, \ldots, i-1,r)\sum_{s=1}^{j-1} (-1)^{s+j} D^-(m+1, \ldots, \widehat{m+s},\ldots, m+j)y_{r,m+s}\]
is a sum of multiples of $y_{r,m+s}$, where $r>i$ and $s<j$.

Then $\pi_{I|J}$ can be written as a sum
\begin{equation}\label{rovno}
\begin{aligned} &v\otimes y_{I|J} +\\
&+\sum_{r_1=i_1+1}^m v\frac{D^+(1, \ldots, i_1-1, r_1)}{D^+(1, \ldots, i_1-1, i_1)} \otimes y_{r_1,m+j_1}\otimes y_{i_2, m+j_2} \otimes \cdots \otimes y_{i_k, m+j_k}\\
&+\ldots\\
&+\sum_{r_k=i_k+1}^m v\frac{D^+(1, \ldots, i_k-1, r_k)}{D^+(1, \ldots, i_k-1, i_k)} \otimes y_{i_1,m+j_1}\otimes \cdots \otimes y_{i_{k-1}, m+j_{k-1}} \otimes y_{r_k, m+j_k}\\
&+\sum_{s_1=1}^{j_1-1} (-1)^{s_1+j_1} v\frac{D^-(m+1, \ldots, \widehat{m+s_1}, \ldots, m+j_1)}{D^-(m+1, \ldots, m+j_1-1)} \\
& \hskip1in \otimes y_{i_1,m+s_1}\otimes y_{i_2, m+j_2} \otimes \cdots \otimes y_{i_k, m+j_k}\\
&+\ldots\\
&+\sum_{s_k=1}^{j_k-1} (-1)^{s_k+j_k} v\frac{D^-(m+1, \ldots, \widehat{m+s_k}, \ldots, m+j_k)}{D^-(m+1, \ldots, m+j_k-1)} \\
& \hskip1in \otimes y_{i_1,m+j_1}\otimes \cdots \otimes y_{i_{k-1}, m+j_{k-1}} 
\otimes y_{i_k, m+s_k}\\
&+F_{I|J},\end{aligned}
\end{equation}
where $F_{I|J}$ is a sum of multiples of various $y_{K|L}$, where at least two entries in $(K|L)$ differ from the corresponding entries in $(I|J)$. 

Next, we compute the leading parts of images under $\psi_k$ of various summands appearing in the equation (\ref{rovno}).
Since the coefficients at $y_{K|L}$ appearing in $F_{I|J}$ contain a product of at least two different expressions $D^-(m+1, \ldots, m+l-1)$ in their denominators, we infer that
%$\psi_k(F_{I|J})$ does not contain any term that is a scalar multiple of vector of type $v\otimes y_{M|N}$ for any admissible $(M|N)$, and 
\[\lead(\psi_k(F_{I|J}))=0.\] 
%Therefore, as explained earlier, when determining the coefficients at vectors $vy_{M|N}$, we can replace $\psi_k(\pi_{I|J})$ by $\psi_k(\pi_{I|J}-F_{I|J})$.

Since 
\[\begin{aligned}\psi_k(v\otimes y_{I|J})=&(-1)^{k-1}(v)_{i_k,m+j_k}D \otimes y_{i_1,m+j_1}\wedge \cdots \wedge y_{i_{k-1},m+j_{k-1}}\\
&+(-1)^{k-2} v\otimes (y_{i_1,m+j_1})_{i_k,m+j_k}D \wedge \cdots \wedge y_{i_{k-1},m+j_{k-1}}\\
&+(-1)^{k-3} v\otimes y_{i_1,m+j_1}\wedge (y_{i_2,m+j_2})_{i_k,m+j_k}D \wedge \cdots \wedge y_{i_{k-1},m+j_{k-1}}+\ldots \\
&+v\otimes y_{i_1,m+j_1} \wedge \cdots \wedge (y_{i_{k-1},m+j_{k-1}})_{i_k,m+j_k}D,
\end{aligned}\] 
using Corollary 2.20 and Lemma 2.1 of \cite{fm-p1} we compute 
\[\begin{aligned}\lead(\psi_k(v\otimes y_{I|J}))=&(\lambda^+_{i_k}+\lambda^-_{j_k}) v\otimes y_{i_1,m+j_1}\wedge \cdots \wedge y_{i_{k-1},m+j_{k-1}}\wedge y_{i_k, m+j_k}\\
&+(-1)^{k-2} v\otimes (y_{i_1,m+j_k}\wedge y_{i_k,m+j_1}) \wedge \cdots \wedge y_{i_{k-1},m+j_{k-1}}\\
&+(-1)^{k-3} v \otimes y_{i_1,m+j_1}\wedge (y_{i_2,m+j_k}\wedge y_{i_k,m+j_2})\wedge\ldots \wedge y_{i_{k-1},m+j_{k-1}} \\
&+\ldots +v\otimes y_{i_1,m+j_1} \wedge \cdots \wedge (y_{i_{k-1},m+j_k} \wedge y_{i_k,m+j_{k-1}})\\
=&(\lambda^+_{i_k}+\lambda^-_{j_k}) v \otimes \overline{y}_{I|J}+v\otimes \overline{y}_{I|j_k j_2 \cdots j_{k-1} j_1} +v\otimes\overline{y}_{I|j_1j_kj_3\ldots j_{k-1}j_2}\\
&+\ldots +v\otimes \overline{y}_{I|j_1\cdots j_{k-2}j_k j_{k-1}}.
\end{aligned}\]

We have 
\[\begin{aligned}&\lead(\psi_k(\sum_{r_1=i_1+1}^m v\frac{D^+(1, \ldots, i_1-1, r_1)}{D^+(1, \ldots, i_1-1, i_1)} \otimes y_{r_1,m+j_1}\otimes y_{i_2, m+j_2} 
\otimes \cdots \otimes y_{i_k, m+j_k}))\\
&=\lead((-1)^{k-1}\sum_{r_1=i_1+1}^m (v\frac{D^+(1, \ldots, i_1-1, r_1)}{D^+(1, \ldots, i_1-1, i_1)})_{i_k,m+j_k}D \otimes \\
&\hskip1.5in y_{r_1,m+j_1}\wedge y_{i_2, m+j_2} 
\wedge \cdots \wedge y_{i_{k-1}, m+j_{k-1}}).\end{aligned}\]

Using Corollary 2.20 and Lemma 2.10 of \cite{fm-p1}, we argue that the terms contributing to the leading part of the last expression correspond to 
\[D^+(1, \ldots, i_1-1, r_1)_{i_k,m+j_k}D= \sum_{a=1}^m D^+(1, \ldots, i_1-1, a)y_{a,m+j_k},\]
provided $i_k=r_1$ (which implies $i_k>i_1$), and the only summand contributing to the leading part corresponds to $a=i_1$. 
Therefore the last expression equals
\[\begin{aligned} 
&(-1)^{k-1} \delta_{i_k>i_1} v \otimes y_{i_1,m+j_k}\wedge y_{i_k,m+j_1}\wedge y_{i_2,m+j_2} \wedge \cdots \wedge y_{i_{k-1},m+j_{k-1}} \\
&=- \delta_{i_k>i_1} v \otimes \overline{y}_{I|j_k j_2 \cdots j_{k-1} j_1},
\end{aligned}\]
where $\delta_{i_k>i_1}=1 $ if $i_k>i_1$, and vanishes otherwise.

Analogous formulae are valid for the summands corresponding to sums involving $r_2$ up to $r_{k-1}$ and the last one is 
\[\begin{aligned}&\lead(\psi_k(\sum_{r_{k-1}=i_{k-1}+1}^m v\frac{D^+(1, \ldots, i_{k-1}-1, r_{k-1})}{D^+(1, \ldots, i_{k-1}-1, i_{k-1})} \\
&\hskip1.3in \otimes y_{i_1,m+j_1}\otimes \cdots \otimes y_{r_{k-1}, m+j_{k-1}} \otimes y_{i_k, m+j_k}))\\
&=-\delta_{i_k>i_{k-1}} v \otimes \overline{y}_{I|j_1\cdots j_{k-2}j_kj_{k-1}}.
\end{aligned}\]

The following sum corresponding to $r_k$ behaves differently. Using Lemma 2.10 of \cite{fm-p1} again, we get 
\[\begin{aligned}&\lead(\psi_k(\sum_{r_k=i_k+1}^m v\frac{D^+(1, \ldots, i_k-1, r_k)}{D^+(1, \ldots, i_k-1, i_k)} \\
&\hskip1in \otimes y_{i_1,m+j_1}\otimes \cdots \otimes y_{i_{k-1}, m+j_{k-1}} \otimes y_{r_k, m+j_k}))\\
&=\lead((-1)^{k-1}\sum_{r_k=i_k+1}^m (v\frac{D^+(1, \ldots, i_k-1, r_k)}{D^+(1, \ldots, i_k-1, i_k)})_{r_k,m+j_k}D \otimes \\
&\hskip1.5in y_{i_1,m+j_1}\wedge y_{i_2, m+j_2} \wedge \cdots \wedge y_{i_{k-1}, m+j_{k-1}})\\
&=(-1)^{k-1} (m-i_k) v \otimes y_{i_k,m+j_k}\wedge y_{i_1,m+j_1}\wedge y_{i_2,m+j_2} \wedge \cdots \wedge y_{i_{k-1},m+j_{k-1}}\\
&=(m-i_k) v \otimes \overline{y}_{I|J}.
\end{aligned}\]

We have
\[\begin{aligned}&\lead(\psi_k(\sum_{s_1=1}^{j_1-1} (-1)^{s_1+j_1} v\frac{D^-(m+1, \ldots, \widehat{m+s_1}, \ldots, m+j_1)}{D^-(m+1, \ldots, m+j_1-1)} \\
&\otimes y_{i_1,m+s_1}\otimes y_{i_2, m+j_2} \otimes \cdots \otimes y_{i_k, m+j_k}))\\
&=\lead(\sum_{s_1=1}^{j_1-1} (-1)^{s_1+j_1+k-1} (v\frac{D^-(m+1, \ldots, \widehat{m+s_1}, \ldots, m+j_1)}{D^-(m+1, \ldots, m+j_1-1)})_{i_k,m+j_k}D \\
&\otimes y_{i_1,m+s_1}\wedge y_{i_2, m+j_2} \wedge \cdots \wedge y_{i_{k-1}, m+j_{k-1}}).\end{aligned}\]

Using Corollary 2.20 and Lemma 2.13 of \cite{fm-p1}, we argue that the terms contributing to the leading part of the last expression correspond to 
\[\begin{aligned}&D^-(m+1, \ldots, \widehat{m+s_1}, \ldots, m+j_1)_{i_k,m+j_k}D= \\
&D^-(m+j_k, m+2, \ldots, \widehat{m+s_1}, \ldots, m+j_1)y_{i_k,m+1}\\
&+ D^-(m+1,m+j_k, \ldots, \widehat{m+s_1}, \ldots, m+j_1)y_{i_k,m+2}\\
&+\ldots \\
&+D^-(m+1, \ldots, \widehat{m+s_1}, \ldots, m+j_1-1, m+j_k)y_{i_k, m+j_1}\end{aligned}\]
provided $j_k=s_1$ (which implies $j_k<j_1$), and the only term contributing to the leading part corresponds to the last summand
\[D^-(m+1, \ldots, \widehat{m+j_k}, \ldots, m+j_1-1, m+j_k)y_{i_k, m+j_1}.\]
Since 
\[\begin{aligned}&D^-(m+1, \ldots, \widehat{m+j_k}, \ldots, m+j_1-1, m+j_k)\\
&=(-1)^{j_1+j_k-1} D^-(m+1, \ldots, m+j_1-1)\end{aligned},\]
we infer that 
\[\begin{aligned}&\lead(\psi_k(\sum_{s_1=1}^{j_1-1} (-1)^{s_1+j_1} v\frac{D^-(m+1, \ldots, \widehat{m+s_1}, \ldots, m+j_1)}{D^-(m+1, \ldots, m+j_1-1)} \\
&\otimes y_{i_1,m+s_1}\otimes y_{i_2, m+j_2} \otimes \cdots \otimes y_{i_k, m+j_k}))\\
&=(-1)^{k}  \delta_{j_k<j_1} v \otimes y_{i_k, m+j_1}\wedge y_{i_1,m+j_k}\wedge y_{i_2, m+j_2} \wedge \cdots \wedge y_{i_{k-1}, m+j_{k-1}}\\
&=- \delta_{j_k<j_1} v \otimes \overline{y}_{I|j_k j_2 \ldots j_{k-1} j_1}.
\end{aligned}\]

Analogous formulae are valid for the summands corresponding to sums involving $s_2$ up to $s_{k-1}$ and the last one is 
\[\begin{aligned}&\lead(\psi_k(\sum_{s_{k-1}=1}^{j_{k-1}-1} (-1)^{s_{k-1}+j_{k-1}} v\frac{D^-(m+1, \ldots, \widehat{m+s_{k-1}}, m+j_{k-1})}{D^-(m+1, \ldots, m+j_{k-1}-1)}\\
&\otimes y_{i_1,m+j_1}\otimes \cdots \otimes y_{i_{k-1}, m+s_{k-1}} \otimes y_{i_k, m+j_k}))\\
&=- \delta_{j_k<j_{k-1}} v \otimes \overline{y}_{I|j_1\cdots j_{k-2}j_kj_{k-1}}.
\end{aligned}\]

The last sum corresponding to $s_k$ behaves differently. We have 
\[\begin{aligned}&\lead(\psi_k(\sum_{s_{k}=1}^{j_{k}-1} (-1)^{s_{k}+j_{k}} v\frac{D^-(m+1, \ldots, \widehat{m+s_{k}}, \ldots, m+j_{k})}{D^-(m+1, \ldots, m+j_{k}-1)}\\
&\otimes y_{i_1,m+j_1}\otimes \cdots \otimes y_{i_{k-1}, m+j_{k-1}} \otimes y_{i_k, m+s_k}))\\
&=\lead(\sum_{s_k=1}^{j_k-1} (-1)^{s_k+j_k+k-1} (v\frac{D^-(m+1, \ldots, \widehat{m+s_k}, \ldots, m+j_k)}{D^-(m+1, \ldots, m+j_k-1)})_{i_k,m+s_k}D \\
&\otimes y_{i_1,m+j_1}\wedge y_{i_2, m+j_2} \wedge \cdots \wedge y_{i_{k-1}, m+j_{k-1}}).\end{aligned}\]

Using Corollary 2.20 and Lemma 2.13 of \cite{fm-p1}, we derive that the terms contributing to the leading part of the last expression correspond to 

\[\begin{aligned}&D^-(m+1, \ldots, \widehat{m+s_k}, \ldots, m+j_k)_{i_k,m+s_k}D= \\
&D^-(m+s_k, m+2, \ldots, \widehat{m+s_k}, \ldots, m+j_k)y_{i_k,m+1}\\
&+ D^-(m+1,m+s_k, \ldots, \widehat{m+s_k}, \ldots, m+j_k)y_{i_k,m+2}\\
&+\ldots \\
&+D^-(m+1, \ldots, \widehat{m+s_k}, \ldots, m+j_k-1, m+s_k){y_{i_k, m+j_k}}\end{aligned}\]
for each $s_k$, and the only term contributing to the leading part corresponds to the last summand
\[D^-(m+1, \ldots, \widehat{m+s_k}, \ldots, m+j_k-1, m+s_k)y_{i_k, m+j_k}.\]
Since 
\[\begin{aligned}&D^-(m+1, \ldots, \widehat{m+s_k}, \ldots, m+j_k-1, m+s_k)\\
&=(-1)^{s_k+j_k-1} D^-(m+1, \ldots, m+j_k-1),\end{aligned}\]
we infer that 
\[\begin{aligned}&\lead(\psi_k(\sum_{s_{k}=1}^{j_{k}-1} (-1)^{s_{k}+j_{k}} v\frac{D^-(m+1, \ldots, \widehat{m+s_{k}}, \ldots, m+j_{k})}{D^-(m+1, \ldots, m+j_{k}-1)}\\
&\otimes y_{i_1,m+j_1}\otimes \cdots \otimes y_{i_{k-1}, m+j_{k-1}} \otimes y_{i_k, m+s_k}))\\
&=(-1)^{k} (j_k-1) v \otimes y_{i_k, m+j_k}\wedge y_{i_1,m+j_1}\wedge y_{i_2, m+j_2} \wedge \cdots \wedge y_{i_{k-1}, m+j_{k-1}}\\
&= - (j_k-1) v \otimes \overline{y}_{I|J}.
\end{aligned}\]

Recall that Definition 1.1 of \cite{fm-p1} defines the expression $\omega_{ij}$ as 
\begin{equation}\label{2}
\omega_{ij}=\omega_{ij}(\lambda)=\lambda^+_i+\lambda^-_j+m+1-i-j.
\end{equation}
Recalling the definition of $\pi_{I|J}$ given in (\ref{piIJ}), we are ready to state the following proposition.

\begin{pr}\label{tool}
Let $(I|J)$ be admissible of length $k$. Then 
\[\begin{aligned}\psi_k(\pi_{I|J})= &\omega_{i_kj_k} \overline{\pi}_{I|J} + (1-\delta_{i_k>i_1} -\delta_{j_k<j_1}) \overline{\pi}_{I|j_kj_2\cdots j_{k-1}j_1} \\
&+(1-\delta_{i_k>i_2} -\delta_{j_k<j_2}) \overline{\pi}_{I|j_1j_k\cdots j_{k-1}j_2} + \ldots \\
&+(1-\delta_{i_k>i_{k-1}} -\delta_{j_k<j_{k-1}}) \overline{\pi}_{I|j_1\cdots j_k j_{k-1}}.\end{aligned}\]
\end{pr}
\begin{proof}
We have computed earlier that 
\[\begin{aligned}\lead(\psi_k(\pi_{I|J}))=&(\lambda^+_{i_k}+\lambda^-_{j_k}+m-i_k-j_k+1) v\otimes \overline{y}_{I|J} \\
&+(1-\delta_{i_k>i_1} -\delta_{j_k<j_1}) v \otimes \overline{y}_{I|j_kj_2\cdots j_{k-1}j_1}\\
&+(1-\delta_{i_k>i_2} -\delta_{j_k<j_2}) v \otimes \overline{y}_{I|j_1j_k\cdots j_{k-1}j_2} + \ldots \\
&+(1-\delta_{i_k>i_{k-1}} -\delta_{j_k<j_{k-1}}) v\otimes \overline{y}_{I|j_1\cdots j_k j_{k-1}}.\end{aligned}\]

Since $\pi_{I|J}$ is a $G_{ev}$-primitive vector and $\psi_k$ is a $G_{ev}$-morphism, we infer that $\psi_k(\pi_{I|J})$ is a $G_{ev}$-primitive vector. By Proposition \ref{genbasis}, it is a linear combination of vectors $\overline{\pi}_{M|N}$ for admissible $(M|N)$ such that $\cont(M|N)=\cont(I|J)$. Since the leading part of $\psi_k(\pi_{I|J})$ is the same linear combination of leading parts $v\otimes \overline{y}_{M|N}$ of $\overline{\pi}_{M|N}$, the statement follows from the above description of $\lead(\psi_k(\pi_{I|J}))$.
\end{proof}

As a consequence of the proposition, the map $\psi_k:Q_k\to \overline{Q}_k$ restricts to $\psi_k:S_k\to \overline{S}_k$, 
where $S_k$ is a span of all vectors $\pi_{I|J}$ for admissible $(I|J)$ of length $k$ and 
$\overline{S}_k$ is a span of all vectors $\overline{\pi}_{I|J}$ for admissible $(I|J)$ of length $k$. 

We illustrate this proposition on the following example that will be used later.

\begin{ex}\label{ex1}
Let $m=n=3$ and a weight $\lambda$ is such that $\lambda^+_1>\lambda^+_2>\lambda^+_3$ and $\lambda^-_1>\lambda^-_2>\lambda^-_3$. Assume that $I=J=(123)$, hence $\lambda$ is $(I|J)$-robust. 
Then 
\[\mathfrak{B}=\{\vec{b}_1=\overline{\pi}_{123,123},\vec{b}_2=\overline{\pi}_{123,213},\vec{b}_3=\overline{\pi}_{123,132},\vec{b}_4=\overline{\pi}_{123,312},
\vec{b}_5=\overline{\pi}_{123,231},\vec{b}_6=\overline{\pi}_{123,321}\}\]
is a basis of all even-primitive vectors of weight $\lambda_{I|J}$ in $F_k$.
The set $\mathfrak{A}=$
\[\begin{aligned}\{&\vec{a}_1=\pi_{123,123}, \vec{a}_2=\pi_{123,213},\vec{a}_3=\pi_{123,132},\vec{a}_4=\pi_{123,312},\vec{a}_5=\pi_{123,231},\vec{a}_6=\pi_{123,321},\\
&\vec{a}_7=\pi_{213,213}, \vec{a}_8=\pi_{213,123},\vec{a}_9=\pi_{213,312},\vec{a}_{10}=\pi_{213,132},\vec{a}_{11}=\pi_{213,321},\vec{a}_{12}=\pi_{213,231}, \\
&\vec{a}_{13}=\pi_{132,132}, \vec{a}_{14}=\pi_{132,231},\vec{a}_{15}=\pi_{132,123},\vec{a}_{16}=\pi_{132,321},\vec{a}_{17}=\pi_{132,213},\vec{a}_{18}=\pi_{132,312}, \\
&\vec{a}_{19}=\pi_{312,312}, \vec{a}_{20}=\pi_{312,321},\vec{a}_{21}=\pi_{312,213},\vec{a}_{22}=\pi_{312,231},\vec{a}_{23}=\pi_{312,123},\vec{a}_{24}=\pi_{312,132}, \\
&\vec{a}_{25}=\pi_{231,231}, \vec{a}_{26}=\pi_{231,132},\vec{a}_{27}=\pi_{231,321},\vec{a}_{28}=\pi_{231,123},\vec{a}_{29}=\pi_{231,312},\vec{a}_{30}=\pi_{231,213}, \\
&\vec{a}_{31}=\pi_{321,321}, \vec{a}_{32}=\pi_{321,312},\vec{a}_{33}=\pi_{321,231},\vec{a}_{34}=\pi_{321,213},\vec{a}_{35}=\pi_{321,132},\vec{a}_{36}=\pi_{321,123}\}
\end{aligned}\]
is linearly independent and spans the space $A$ of even-primitive vectors of weight $\lambda_{I|J}$ in $T_k$.
For an explanation how was the set $\mathfrak{A}$ ordered see Lemma \ref{surj} and Example \ref{ex2} later.

The matrix of $\psi_k$ with respect to the bases $\mathfrak{A}$ and $\mathfrak{B}$ is given as
\[\left(\begin{matrix}
\omega_{33}&0&-1&0&0&-1&-\omega_{33}&0&1&0&0&1\\
0&\omega_{33}&0&-1&-1&0&0&-\omega_{33}&0&1&1&0\\
0&0&\omega_{32}&0&-1&0&0&0&-\omega_{32}&0&1&0\\
0&0&0&\omega_{32}&0&-1&0&0&0&-\omega_{32}&0&1\\
0&0&0&0&\omega_{31}&0&0&0&0&0&-\omega_{31}&0\\
0&0&0&0&0&\omega_{31}&0&0&0&0&0&-\omega_{31}
\end{matrix}\right.\]

\[\begin{matrix} 
-\omega_{22}&1&-1&0&0&0&\omega_{22}&-1&1&0&0&0\\
0&-\omega_{21}&0&0&-1&0&0&\omega_{21}&0&0&1&0\\
0&0&-\omega_{23}&1&0&0&0&0&\omega_{23}&-1&0&0\\
0&0&0&-\omega_{22}&0&-1&0&0&0&\omega_{22}&0&1\\
0&0&0&0&-\omega_{23}&1&0&0&0&0&\omega_{23}&-1\\
0&0&0&0&0&-\omega_{22}&0&0&0&0&0&\omega_{22}
\end{matrix}\]

\[\left.\begin{matrix}
\omega_{11}&1&0&0&0&1&-\omega_{11}&-1&0&0&0&-1\\
0&\omega_{12}&0&1&0&0&0&-\omega_{12}&0&-1&0&0\\
0&0&\omega_{11}&1&1&0&0&0&-\omega_{11}&-1&-1&0\\
0&0&0&\omega_{13}&0&0&0&0&0&-\omega_{13}&0&0\\
0&0&0&0&\omega_{12}&1&0&0&0&0&-\omega_{12}&-1\\
0&0&0&0&0&\omega_{13}&0&0&0&0&0&-\omega_{13}
\end{matrix}\right).\]

Here $\omega_{33}=\lambda^+_3+\lambda_3^- -2$, $\omega_{32}=\lambda^+_3+\lambda_2^- -1$, $\omega_{31}=\lambda^+_3+\lambda_1^-$, 
$\omega_{23}=\lambda^+_2+\lambda_3^- -1$, $\omega_{22}=\lambda^+_2+\lambda_2^- $, $\omega_{21}=\lambda^+_2+\lambda_1^- +1$, 
$\omega_{13}=\lambda^+_1+\lambda_3^-$, $\omega_{12}=\lambda^+_1+\lambda_2^- +1$ and $\omega_{11}=\lambda^+_1+\lambda_1^- +2$. 
\end{ex}

\subsection{Map $\psi_{\mu}$ and the strong linkage}

Assume $(I|J)$ is admissible of length $k$, and denote the weight $\lambda_{I|J}$ by $\mu$. Denote by $S_{\mu}$ or $S_{I|J}$ 
the span of all even-primitive vectors of weight $\mu$ in $T_k$ and 
by $\overline{S}_{\mu}$ or $\overline{S}_{I|J}$ the span of all even-primitive vectors of weight $\mu$ in $F_k$. 
Consider a restriction $\psi_{\mu}$ of the map $\psi_k$ on $S_{\mu}$. By Lemma \ref{evmorph}, $\psi_{\mu}$ maps $S_{\mu}$ to $\overline{S}_{\mu}$.

\begin{pr}\label{comp}
Let $\lambda$ be a dominant weight of $G$, $(I|J)$ be admissible of length $k$ such that $\mu=\lambda_{I|J}$ is dominant. 
Then  $\mu$ is strongly linked to  $\lambda$ if and only if 
the map $\psi_{\mu}: S_{\mu} \to \overline{S}_{\mu}$ is not surjective.
If both $\lambda$ and $\mu$ are polynomial weights, then $L_G(\mu)$ is a composition factor of $\nabla(\lambda)$ if and only if 
the map $\psi_{\mu}: S_{\mu} \to \overline{S}_{\mu}$ is not surjective.
\end{pr}
\begin{proof}
It is well-known that the category of rational $G_{ev}$-modules is semisimple. Its simple objects $L_{G_{ev}}(\mu)$ are indexed by dominant weights $\mu$, and are generated by $G_{ev}$-primitive vectors
$v_{\mu}$.
On the other hand, the supermodule $H^0_G(\lambda)$ need not be semisimple. 

Using Poincare-Birkhoff-Witt theorem, we order generators $e_{ij}$ of the distribution algebra $\Dist(G)$ (universal enveloping algebra of the Lie superalgebra $\mathfrak{gl}(m|n)$)  
in the following way. First we list odd elements $e_{ij}$, where $m+1\leq j \leq m+n$ and $1\leq i \leq m$, followed by even elements $e_{ij}$, and the odd elements $e_{ji}$, where $1\leq i\leq m$ and $m+1\leq j \leq m+n$, come last. 

Consider the filtration 
\[0\subset M_0\subset M_1 \subset \ldots \subset M_{mn}=H^0_G(\lambda)\]
of $H^0_G(\lambda)$ by $G$-supermodules $M_i=\langle F_1, \ldots, F_i\rangle$ generated by floors $F_1, \ldots, F_i$.

The supermodule $L_G(\mu)$ is a composition factor of $H^0_G(\lambda)$ if and only if it is a composition factor of 
$M_{k}/M_{k-1}$. Using the above ordering of generators of $\Dist(G)$, we infer that this happens if and only if there is a $G_{ev}$-primitive vector in $F_k$ of weight $\mu$ that does not belong to the $G$-subsupermodule $M_{k-1}$. 

The map $\psi_k: T_k \to F_k$ factors through the natural projection 
\[\proj:T_k=V\otimes Y^{\otimes k} \to V \otimes (\wedge^{k-1} Y) \otimes Y = F_{k-1}\otimes Y =\tilde{T}_k\] which is a surjective $G_{ev}$-morphism.
Denote by $\proj_{\mu}$ the restriction of $\proj$ to the weight spaces corresponding to $\mu$, 
by $\tilde{S}_{\mu}$ the span of even-primitive vectors of weight $\mu$ in $\tilde{T}_k$ and by $\tilde{\psi}_{\mu}$ the induced map satisfying 
$\psi_{\mu}=\tilde{\psi}_{\mu}\circ proj_{\mu}$. 

From the definitions of the supermodule $M_{k-1}$ and the above factorization of $\psi_k$, it is clear that $\psi_k(T_k)\subseteq M_{k-1}\cap F_k$. To see the reverse inclusion, consider one of the generators $w \in M_{k-1}\cap F_l$ of $M_{k-1}$, where $l<k$. Denote by $W$ the $G$-supermodule generated by $w$. 

Using the ordering of superderivations $_{ij}D$ corresponding to the above ordering of $e_{ji}$, we infer that every $w'\in W\cap F_k$ is a sum of terms of type $u=(w'')_{i_1,j_1}D\ldots _{i_{k-l},j_{k-l}}D$, where 
$w''\in W\cap F_l$, the superderivations $_{i_s,j_s}D$ are odd, and $1\leq i_s\leq m<j_s\leq m+n$ for each 
$s=1, \ldots, k-l$.
Since \[u'=(w'')_{i_1,j_1}D\ldots _{i_{k-l-1},j_{k-l-1}}D\in F_{k-1}\]
and
$u=(u')_{i_{k-l}j_{k-l}}D=\tilde{\psi}_k(u'\otimes y_{i_{k-l},j_{k-l}}) \in \psi_k(T_k)$, we conclude that 
$W\cap F_k \subseteq \psi_k(T_k)$ and hence $M_{k-1}\cap F_k \subseteq \psi_k(T_k)$.  Therefore $M_{k-1}\cap F_k=\psi_k(T_k)$.

By Lemma \ref{evmorph}, the map $\psi_k$ sends $G_{ev}$-primitive vectors of weight $\mu$ in $T_k$ to $G_{ev}$-primitive vectors in $F_k$. 
Therefore, there is a $G_{ev}$-primitive vector of weight $\mu$ in $F_k$ that does not belong to $M_{k-1}$ if and only if the map $\psi_{\mu}:S_{\mu} \to \overline{S}_{\mu}$ is not surjective.

The last statement now follows from Corollary 7.2 of \cite{fm-p2}.
\end{proof}

Explicit elements of $S_{I|J}$, and an explicit basis of $\overline{S}_{I|J}$, for $(I|J)$-robust weights $\lambda$ were given in \cite{fm-p1}, and in the general case in 
\cite{fm-p2}.

Using these and our description of the matrix of $\psi_k$, we determine the necessary condition when 
the map $\psi_{\mu}: S_{\mu} \to \overline{S}_{\mu}$ is not surjective. Using Proposition \ref{comp}, we then obtain a description of simple composition factors of
$H^0_G(\lambda)$ and $\nabla(\lambda)$.
We start first with the case of a robust weight and deal with the general case later.

In what follows, we abuse the notation and denote various restrictions of the map $\psi_k$ (and $\psi_{\mu}$) just by $\psi_k$. We indicate the domain and codomain of the map
$\psi_k$ every time it is required to avoid confusion.

\section{Odd linkage for robust weights}\label{sec4}
In this section, we {\it assume that the weight $\lambda$ is $(I|J)$-robust}. 

\subsection{$\lambda$ is $(I|J)$-robust, entries in $I$ are distinct, and entries in $J$ are distinct}
Assume first that all entries in $I$ are distinct, and all entries in $J$ are distinct. In this case, the condition that $\lambda$ is $(I|J)$-robust is equivalent to the condition that
all entries $\lambda^+_i$ for $i\in I$, and all entries $\lambda^-_j$ for $j\in J$ are distinct.
We will remove the assumption that all entries in $I$ are distinct, and all entries in $J$ are distinct in the next subsection.

Let $I_0=(i_1 < i_2 < \ldots < i_k)$ be a multi-index of the same content as $I$ and $J_0=(j_1 < j_2 < \ldots < j_k)$ be a multi-index of the same content as $J$.
Then $\overline{S}_{I|J}$ is the span of vectors $\overline{\pi}_{I_0|L}$ for all permutations $L$ of $J_0$. 

Denote by $S_J$ the set of all multi-indices $L$ of content $\cont(J)$. We impose the reverse Semitic lexicographic order $<$ on the set $S_J$ . This means that we compare entries of two elements of $S_J$ by reading from right to left and we impose the reverse order on the individual symbols. For example, if $J=(3<2<1)$, then the order $<$ on $S_J$
is $123<213<132<312<231<321$. The order $L_1< L_2 < \ldots < L_{k!}$ on $S_J$ induces the corresponding order $<$ on basis elements $\overline{\pi}_{I_0|L}$ of $\overline{S}_{I|J}$
given as $\overline{\pi}_{I_0|L_1} < \overline{\pi}_{I_0|L_2} < \ldots <\overline{\pi}_{I_0|L_{k!}}$.

For a permutation $\eta\in \Sigma_k$ of the set $\{1, \ldots, k\}$ and $L=(l_1, \ldots, l_k)$, we denote $\eta L=(l_{\eta(1)}, \ldots, l_{\eta(k)})$. 
The vector space $S_{I|J}$ contains a subspace that is a direct sum of spaces $S_{\eta I_0|J}$ for all $\eta\in \Sigma_k$, where each $S_{\eta I_0|J}$ is the spans of all vectors $\pi_{\eta I_0|L}$
for all multi-indices $L$ of content $\cont(J)$. The dimension of $S_{\eta I_0|J}$ is $k!$ and the matrix of $\psi_k:S_{\eta I_0|J}\to \overline{S}_{I|J}$ for bases consisting of vectors $\pi_{\eta I_0|L}$ and $\overline{\pi}_{I_0|L}$, respectively, is 
a square matrix of dimension $k!$.

Related to the order $<$ on $S_{J}$ and the permutation $\eta\in \Sigma_k$, we define an order $<^{\eta}$ on the basis elements $\pi_{\eta I_0|L}$ of $S_{\eta I_0|J}$ as follows.
If $L_1< L_2 \ldots < L_{k!}$ is the listing of the elements of $S_J$ according to the order $<$, then the order $<^\eta$ on the elements $\pi_{\eta I_0|L}$ is given as
$\pi_{\eta I_0|\eta L_1}<^\eta \pi_{\eta I_0|\eta L_2}<^\eta \ldots <^\eta \pi_{\eta I_0|\eta L_{k!}}$. This is compatible with the order $<$ on basis elements $\overline{\pi}_{I_0|L}$ of $\overline{S}_{I|J}$ since $\overline{\pi}_{\eta I_0|\eta L_s}=\pm \overline{\pi}_{I_0|L_s}$.

When we compare various even-primitive vectors in $\overline{Q}_k$ or $\overline{S}_k$, it is useful to mention that $\overline{\pi}_{I_0,L} = (-1)^\eta \overline{\pi}_{\eta I_0,\eta L}$.

Recall the definition of $\omega_{i,j}$ given in (\ref{2}).

\begin{lm}\label{surj}
Assume that all entries in $I$ are distinct, all entries in $J$ are distinct, and $\lambda$ is $(I|J)$-robust. Then the matrix of the map $\psi_k:S_{\eta I_0|J}\to \overline{S}_{I|J}$ for the bases consisting of vectors  $\pi_{\eta I_0|\eta L}$ ordered by $<^\eta$, and $\overline{\pi}_{I_0|L}$ ordered by $<$, is an upper-triangular matrix. Its diagonal entry
corresponding to $\pi_{\eta I_0|\eta L}$ and $\overline{\pi}_{I_0|L}$ equals $(-1)^{\eta}\omega_{i_{\eta(k)}, l_{\eta(k)}}$.
\end{lm}
\begin{proof}
By Proposition \ref{tool}, we have 
\[\begin{aligned}\psi_k(\pi_{\eta I_0|\eta L})= &\omega_{i_{\eta(k)}l_{\eta(k)}} \overline{\pi}_{\eta I_0|\eta L} \\
&+ (1-\delta_{i_{\eta(k)}>i_{\eta(1)}} -\delta_{l_{\eta(k)}<l_{\eta(1)}}) \overline{\pi}_{\eta I_0|l_{\eta(k)}l_{\eta(2)}\cdots l_{\eta(k-1)}l_{\eta(1)}}\\
&+ (1-\delta_{i_{\eta(k)}>i_{\eta(2)}} -\delta_{l_{\eta(k)}<l_{\eta(2)}}) \overline{\pi}_{\eta I_0|l_{\eta(1)}l_{\eta(k)}l_{\eta(3)}\cdots l_{\eta(k-1)}l_{\eta(2)}} + \ldots \\
&+ (1-\delta_{i_{\eta(k)}>i_{\eta(k-1)}} -\delta_{l_{\eta(k)}<l_{\eta(k-1)}}) \overline{\pi}_{\eta I_0|l_{\eta(1)}\cdots l_{\eta(k)}l_{\eta(k-1)}}.\end{aligned}\]
Therefore, the diagonal entries of our matrix are the same as described in the statement of the lemma, and it only remains to show that our matrix
is upper-triangular.

Consider the general term 
\[ (1-\delta_{i_{\eta(k)}>i_{\eta(t)}} -\delta_{l_{\eta(k)}<l_{\eta(t)}}) 
\overline{\pi}_{\eta I_0|l_{\eta(1)}\cdots l_{\eta(t-1)} l_{\eta(k)} l_{\eta(t+1)}\cdots l_{\eta(k-1)}l_{\eta(t)}}\]
in the above formula for $\psi_k(\pi_{\eta I_0|\eta L})$.
Let $\gamma=(\eta(t)\eta(k))$ be the transposition switching entries in positions $\eta(t)$ and $\eta(k)$ and $K=\gamma.L$. 

Then $\eta K=(l_{\eta(1)}\cdots l_{\eta(t-1)} l_{\eta(k)} l_{\eta(t+1)}\cdots l_{\eta(k-1)}l_{\eta(t)})$ and the above general term becomes
\[ (1-\delta_{i_{\eta(k)}>i_{\eta(t)}} -\delta_{l_{\eta(k)}<l_{\eta(t)}}) \overline{\pi}_{\eta I_0|\eta K}.\]

If $i_{\eta(k)}>i_{\eta(t)}$ and $l_{\eta(k)}>l_{\eta(t)}$,  or $i_{\eta(k)}<i_{\eta(t)}$ and $l_{\eta(k)}<l_{\eta(t)}$, then this term vanishes.

If $i_{\eta(k)}>i_{\eta(t)}$ and $l_{\eta(k)}<l_{\eta(t)}$, then $\eta(t)<\eta(k)$ and $l_{\eta(t)}>l_{\eta(k)}$ imply that
$K=\gamma.L<L$. If $i_{\eta(k)}<i_{\eta(t)}$ and $l_{\eta(k)}>l_{\eta(t)}$, then $\eta(t)>\eta(k)$ and $l_{\eta(t)}<l_{\eta(k)}$ imply again that
$K=\gamma.L<L$. Therefore, in both cases $\pi_{\eta I_0,\eta K} {\lneq}^\eta \pi_{\eta I_0,\eta L}$ and  
$\overline{\pi}_{I_0|K}{\lneq}\overline{\pi}_{I_0|L}$, it implies that our matrix is upper-triangular.
\end{proof}

The above lemma implies immediately that the determinant of the matrix of the map $\psi_k:S_{\eta I_0|J}\to \overline{S}_{I|J}$ is 
$(\prod_{t=1}^k \omega_{i_{\eta(k)},j_t})^{(k-1)!}$.

We illustrate this lemma on the following example.

\begin{ex}\label{ex2}
Let $m=n=3$ and a weight $\lambda$ is such that $\lambda^+_1>\lambda^+_2>\lambda^+_3$ and $\lambda^-_1>\lambda^-_2>\lambda^-_3$. Assume that $I=J=(123)$, hence $\lambda$ is $(I|J)$-robust. 
In the notation of Example \ref{ex1}, we have 
$\mathfrak{B}=\{\vec{b}_1, \vec{b}_2, \vec{b}_3, \vec{b}_4, \vec{b}_5, \vec{b}_6\},$
and 
\[\begin{array}{ll}\mathfrak{A}_{id}=\{\vec{a}_1, \vec{a}_2,\vec{a}_3,\vec{a}_4,\vec{a}_5,\vec{a}_6\}, 
&\mathfrak{A}_{(12)}=\{\vec{a}_7, \vec{a}_8,\vec{a}_9,\vec{a}_{10},\vec{a}_{11},\vec{a}_{12}\} \\
\mathfrak{A}_{(23)}=\{\vec{a}_{13}, \vec{a}_{14},\vec{a}_{15},\vec{a}_{16},\vec{a}_{17},\vec{a}_{18}\}, 
&\mathfrak{A}_{(321)}=\{\vec{a}_{19}, \vec{a}_{20},\vec{a}_{21},\vec{a}_{22},\vec{a}_{23},\vec{a}_{24}\}, \\
\mathfrak{A}_{(231)}=\{\vec{a}_{25}, \vec{a}_{26},\vec{a}_{27},\vec{a}_{28},\vec{a}_{29},\vec{a}_{30}\} \,\, \mbox{and}    
&\mathfrak{A}_{(13)}=\{\vec{a}_{31}, \vec{a}_{32},\vec{a}_{33},\vec{a}_{34},\vec{a}_{35},\vec{a}_{36}\}.
\end{array}\]
The matrices considered in Lemma \ref{surj} can be identified with blocks, formed by consecutive sextuples of columns, of the matrix in Example \ref{ex1}. 
\end{ex}

We say that weights $\lambda$ and $\mu$ are {\it simply-odd-linked}, and write $\lambda\sim_{sodd} \mu$, if $\lambda_{ij}=\mu$ or $\mu_{ij}=\lambda$, and 
$\omega_{ij}(\lambda)=0$ (in this case also $\omega_{ij}(\mu)=0$). We say that $\lambda$ and $\mu$ are {\it odd-linked}, and write $\lambda\sim_{odd} \mu$ if there is a chain of weights such that
$\lambda=\lambda_1\sim_{sodd}\lambda_2 \sim_{sodd} \cdots \sim_{sodd} \lambda_t=\mu$.

Let us note that if a weight $\mu$ of $H^0_G(\lambda)$ is odd-linked to $\lambda$, then $\mu=\lambda_{I|J}$ for some admissible $(I|J)$.

If $\lambda\sim_{odd} \mu$ and $\lambda$, $\mu$, and all intermediate $\lambda_i$ in the above chain are polynomial weights, then we write $\lambda\sim_{podd} \mu$.

\begin{pr}\label{rob-link}
Assume that all entries in $I$ are distinct, all entries in $J$ are distinct, and $\lambda$ is $(I|J)$-robust.
If the simple $G$-supermodule $L_G(\lambda_{I|J})$ is a composition factor of $H^0_G(\lambda)$, then $\lambda \sim_{odd} \lambda_{I|J}$.

If both $\lambda$ and $\lambda_{I|J}$ are polynomial weights and the simple $G$-supermodule $L_G(\lambda_{I|J})$ is a composition factor of $\nabla(\lambda)$, then $\lambda \sim_{podd} \lambda_{I|J}$.

\end{pr}
\begin{proof}
We use Proposition \ref{comp} and assume that the map $\psi_k:S_{I|J} \to \overline{S}_{I|J}$ is not surjective.

Create collections $C_N$, each indexed by a multi-index $N=(n_1, \ldots, n_k)$ of content $\cont(J)$. 
The collection $C_N$ consists of vectors $\pi_{\eta I_0|\eta N}$ for all $\eta\in \Sigma_k$. The images 
$\psi_k(\pi_{\eta I_0|\eta N})$ of elements in the collection $C_N$ expressed as a linear combination of vectors $\overline{\pi}_{I_0|L}$ have the property that 
the coefficient at $\overline{\pi}_{I_0|N}$ equals $\omega_{i_{\eta(k)},n_{\eta(k)}}$, and all other nonzero coefficients appear only at $\overline{\pi}_{I_0|L}$ for $L<N$.
To each collection $C_N$ we assign the set $O_N=\{\omega_{i_{\eta(k)},n_{\eta(k)}}|\eta\in \Sigma_k\}= \{\omega_{i_t,n_t}|t=1, \ldots k\}$.
It is a crucial observation that if every $O_N$ contains a nonzero element, then the map $\psi_k: S_{I|J} \to \overline{S}_{I|J}$ is surjective. This is because 
by Lemma \ref{surj}, we can find a set of $k!$ vectors $\pi_{K|L}$ of weight $\lambda_{I|J}$ such that the matrix of $\psi_k$ restricted on the span of these vectors is invertible.

Therefore, if $\psi_k: S_{I|J} \to \overline{S}_{I|J}$ is not surjective, then there is $N$ such that all elements of $O_N$ are equal to zero. 
In this case we derive that $\lambda$ and $\lambda_{I|J}$ are odd-linked via a sequence 
$\lambda \sim_{sodd} \lambda_{i_1|n_1} \sim_{sodd} \lambda_{i_1i_2|n_1n_2} \sim_{sodd} \cdots \sim_{sodd} \lambda_{I|J}$ because
$\omega_{i_2|n_2}(\lambda_{i_1|n_1})=\omega_{i_2|n_2}(\lambda)=0$ and so on.

The last statement follows from Corollary 7.2 of \cite{fm-p2}.
\end{proof}

\begin{ex}\label{ex3}
Let $m=n=3$ and a weight $\lambda$ is such that $\lambda^+_1>\lambda^+_2>\lambda^+_3$ and $\lambda^-_1>\lambda^-_2>\lambda^-_3$. Assume that $I=J=(123)$, hence $\lambda$ is $(I|J)$-robust. 
In the notation of Example \ref{ex1}, we have $\mathfrak{B}=\{\vec{b}_1, \vec{b}_2, \vec{b}_3, \vec{b}_4, \vec{b}_5, \vec{b}_6\}.$
To illustrate the argument in the previous proposition,  
we determine that collections are
\[\begin{aligned}&C_{123}=\{\vec{a}_1, \vec{a}_7, \vec{a}_{13}, \vec{a}_{19}, \vec{a}_{25},\vec{a}_{31}\}, 
&C_{213}=\{\vec{a}_2, \vec{a}_8, \vec{a}_{14}, \vec{a}_{20}, \vec{a}_{26},\vec{a}_{32}\}, \\
&C_{132}=\{\vec{a}_3, \vec{a}_9, \vec{a}_{15}, \vec{a}_{21}, \vec{a}_{27},\vec{a}_{33}\},
&C_{312}=\{\vec{a}_4, \vec{a}_{10}, \vec{a}_{16}, \vec{a}_{22}, \vec{a}_{28},\vec{a}_{34}\}, \\
&C_{231}=\{\vec{a}_5, \vec{a}_{11}, \vec{a}_{17}, \vec{a}_{23}, \vec{a}_{29},\vec{a}_{35}\} \, \mbox{and} 
&C_{321}=\{\vec{a}_6, \vec{a}_{12}, \vec{a}_{18}, \vec{a}_{24}, \vec{a}_{30},\vec{a}_{36}\},
\end{aligned}\]
while the corresponding sets are 
\[\begin{aligned}&O_{123}=\{\omega_{33}, \omega_{22}, \omega_{11}\}, O_{213}=\{\omega_{33}, \omega_{21}, \omega_{12}\}, &O_{132}=\{\omega_{32}, \omega_{23}, \omega_{11}\}, \\
&O_{312}=\{\omega_{32}, \omega_{21}, \omega_{13}\}, O_{231}=\{\omega_{31}, \omega_{23}, \omega_{12}\} \, \mbox{and} \, &O_{321}=\{\omega_{31}, \omega_{22}, \omega_{13}\}.\end{aligned}\]
\end{ex}

\subsection{General case when the weight $\lambda$ is $(I|J)$-robust}

Let us start with the following three examples.

\begin{ex}\label{ex5}
Let $m=n=3$ and a weight $\lambda$ is such that $\lambda^+_1-1>\lambda^+_2>\lambda^+_3$ and $\lambda^-_1>\lambda^-_2>\lambda^-_3$. Assume that $I=(113)$ and $J=(123)$, hence $\lambda$ is $(I|J)$-robust. 
Then 
\[\mathfrak{B}=\{\vec{b}_1=\overline{\pi}_{113,123}, \vec{b}_2=\overline{\pi}_{113,132},\vec{b}_3=\overline{\pi}_{113,231}\}\]
is a basis of all even-primitive vectors of weight $\lambda_{I|J}$ in $F_k$.
The set $\mathfrak{A}=$
\[\begin{aligned}\{&\vec{a}_1=\pi_{113,123}, \vec{a}_2=\pi_{113,213},\vec{a}_3=\pi_{113,132},\vec{a}_4=\pi_{113,312},\vec{a}_5=\pi_{113,231},\vec{a}_6=\pi_{113,321},\\
&\vec{a}_{7}=\pi_{131,132}, \vec{a}_{8}=\pi_{131,231},\vec{a}_{9}=\pi_{131,123},\vec{a}_{10}=\pi_{131,321},\vec{a}_{11}=\pi_{131,213},\vec{a}_{12}=\pi_{131,312}, \\
&\vec{a}_{13}=\pi_{311,312}, \vec{a}_{14}=\pi_{311,321},\vec{a}_{15}=\pi_{311,213},\vec{a}_{16}=\pi_{311,231},\vec{a}_{17}=\pi_{311,123},\vec{a}_{18}=\pi_{311,132}\}
\end{aligned}\]
is linearly independent and spans the space $A$ of even-primitive vectors of weight $\lambda_{I|J}$ in $T_k$.

The matrix of $\psi_k$ with respect to the bases $\mathfrak{A}$ and $\mathfrak{B}$ computed using Proposition \ref{tool} is given as
\[\left(\begin{matrix}
\omega_{33}&-\omega_{33}&-1&1&1&-1&-\omega_{12}+1&\omega_{11}&-1&0\\
0&0&\omega_{32}&-\omega_{32}&-1&1&0&0&-\omega_{13}+1&\omega_{11}\\
0&0&0&0&\omega_{31}&-\omega_{31}&0&0&0&0
\end{matrix}\right.\]

\[\left.\begin{matrix} 
1&0&\omega_{12}-1&-\omega_{11}&1&0&-1&0\\
0&1&0&0&\omega_{13}-1&-\omega_{11}&0&-1\\
-\omega_{13}+1&\omega_{12}&0&0&0&0&\omega_{13}-1&-\omega_{12}
\end{matrix}\right)\]

The collections are 
\[\begin{aligned}&C_{1}=\{\vec{a}_1, \vec{a}_2, \vec{a}_{7}, \vec{a}_{8}, \vec{a}_{13},\vec{a}_{14}\}, 
&C_{2}=\{\vec{a}_3, \vec{a}_4, \vec{a}_{9}, \vec{a}_{10}, \vec{a}_{15},\vec{a}_{16}\}, \\
&C_{3}=\{\vec{a}_5, \vec{a}_6, \vec{a}_{11}, \vec{a}_{12}, \vec{a}_{17},\vec{a}_{18}\},
\end{aligned}\]
while the corresponding sets are 
\[\begin{aligned}&O_{1}=\{\omega_{33}, \omega_{12}-1, \omega_{11}\}, 
O_{2}=\{\omega_{32}, \omega_{11}, \omega_{13}-1\}, 
&O_{3}=\{\omega_{31}, \omega_{12}, \omega_{13}-1\}.
\end{aligned}\]

If all entries in $O_{1}$ are zeroes, then $\omega_{12}(\lambda_{11})=0$, and 
we get $\lambda\sim_{sodd}\lambda_{11}\sim_{sodd}\lambda_{11|12}\sim_{sodd}\lambda_{113|123}$.

If all entries in $O_{2}$ are zeroes, then $\omega_{13}(\lambda_{11})=0$, and 
we get $\lambda\sim_{sodd}\lambda_{11}\sim_{sodd}\lambda_{11|13}\sim_{sodd}\lambda_{113|132}$.

If all entries in $O_{3}$ are zeroes, then $\omega_{13}(\lambda_{12})=0$, and 
we get $\lambda\sim_{sodd}\lambda_{12}\sim_{sodd}\lambda_{11|23}\sim_{sodd}\lambda_{113|231}$.
\end{ex}

\begin{ex}\label{ex6}
Let $m=n=3$ and a weight $\lambda$ is such that $\lambda^+_1-1>\lambda^+_2>\lambda^+_3$ and $\lambda^-_1-1>\lambda^-_2>\lambda^-_3$. Assume that $I=(113)$ and $J=(112)$, hence $\lambda$ is $(I|J)$-robust. 
Then 
\[\mathfrak{B}=\{\vec{b}_1=\overline{\pi}_{113,121}\}\]
is a basis of all even-primitive vectors of weight $\lambda_{I|J}$ in $F_k$.
The set $\mathfrak{A}=$
\[\begin{aligned}\{&\vec{a}_1=\pi_{113,121}, \vec{a}_2=\pi_{113,211},\vec{a}_3=\pi_{131,112},\vec{a}_4=\pi_{131,211},\vec{a}_5=\pi_{311,112},\vec{a}_6=\pi_{311,121}\}
\end{aligned}\]
is linearly independent and spans the space $A$ of even-primitive vectors of weight $\lambda_{I|J}$ in $T_k$.

The matrix of $\psi_k$ with respect to the bases $\mathfrak{A}$ and $\mathfrak{B}$ computed using Proposition \ref{tool} is given as
\[\begin{pmatrix}\omega_{31}&-\omega_{31}&1-\omega_{12}&\omega_{11}+1&\omega_{12}-1&-\omega_{11}-1\end{pmatrix}\]

There is only one collection 
\[\begin{aligned}&C_{1}=\{\vec{a}_1, \vec{a}_2, \vec{a}_{3}, \vec{a}_{4}, \vec{a}_{5},\vec{a}_{6}\}\end{aligned}\]
and the corresponding set is
\[\begin{aligned}&O_{1}=\{\omega_{31}, \omega_{12}-1, \omega_{11}+1\}.
\end{aligned}\]

If all entries in $O_{1}$ are zeroes, then $\omega_{11}(\lambda_{31})=0$ and $\omega_{12}(\lambda_{31|11})=0$, 
and we get $\lambda\sim_{sodd}\lambda_{31}\sim_{sodd}\lambda_{31|11}\sim_{sodd}\lambda_{311|112}$.
\end{ex}

\begin{ex}\label{ex7}
Let $m=n=3$ and a weight $\lambda$ is such that $\lambda^+_1>\lambda^+_2>\lambda^+_3$ and $\lambda^-_1-2>\lambda^-_2>\lambda^-_3$. Assume that $I=(123)$ and $J=(111)$, hence $\lambda$ is $(I|J)$-robust. 
Then 
\[\mathfrak{B}=\{\vec{b}_1=\overline{\pi}_{123,111}\}\]
is a basis of all even-primitive vectors of weight $\lambda_{I|J}$ in $F_k$.
The set $\mathfrak{A}=$
\[\begin{aligned}\{&\vec{a}_1=\pi_{123,111}, \vec{a}_2=\pi_{213,111},\vec{a}_3=\pi_{132,111},\vec{a}_4=\pi_{312,111},\vec{a}_5=\pi_{231,111},\vec{a}_6=\pi_{321,111}\}
\end{aligned}\]
is linearly independent and spans the space $A$ of even-primitive vectors of weight $\lambda_{I|J}$ in $T_k$.

The matrix of $\psi_k$ with respect to the bases $\mathfrak{A}$ and $\mathfrak{B}$ computed using Proposition \ref{tool} is given as
\[\begin{pmatrix}\omega_{31}&-\omega_{31}&-\omega_{21}-1&\omega_{21}+1&\omega_{11}+2&-\omega_{11}-2\end{pmatrix}\]

There is only one collection 
\[\begin{aligned}&C_{1}=\{\vec{a}_1, \vec{a}_2, \vec{a}_{3}, \vec{a}_{4}, \vec{a}_{5},\vec{a}_{6}\}\end{aligned}\]
and the corresponding set is
\[\begin{aligned}&O_{1}=\{\omega_{31}, \omega_{21}+1, \omega_{11}+2\}.
\end{aligned}\]

If all entries in $O_{1}$ are zeroes, then $\omega_{21}(\lambda_{31})=0$ and $\omega_{11}(\lambda_{32|11})=0$, 
and
we get $\lambda\sim_{sodd}\lambda_{31}\sim_{sodd}\lambda_{32|11}\sim_{sodd}\lambda_{321|111}$.
\end{ex}

Since $\lambda$ is $(I|J)$-robust, by Proposition \ref{basis}, the basis $\mathfrak{B}$ of even-primitive vectors of weight $\lambda_{I|J}$ in $H^0_G(\lambda)$ 
consists of elements $\overline{\pi}_{K|L}$, where $(K|L)$ is admissible and $\cont(K|L)=\cont(I|J)$. 
Denote by $S(\lambda,I,J)$ the set of all multi-indices $L$ such that $(I_0| L)$ is admissible and $\cont(L)=\cont(J)$. Then $S(\lambda, I, J)$ serves as an index set for the basis 
$\mathfrak{B}$.
The previously defined order $<$ on $S_J$ induces the order $<$ on $\mathfrak{B}$ and the order $<^{\eta}$ on $S(\lambda,I,J)$.

For a permutation $\eta\in \Sigma_k$ define $\mathfrak{B}_\eta=\{\pi_{\eta I_0,L}|L\in S(\lambda,I,J)\}$ and $S_\eta=span(\mathfrak{B}_\eta)$.

For $L\in S(\lambda,I,J)$ and $\eta\in \Sigma_k$ define 
\[a_{L,\eta}= |\{t=1,\dots, k|i_{\eta(t)}=i_{\eta(k)} \mbox{ and } l_{\eta(t)}<l_{\eta(k)} \}|\]
and 
\[b_{L,\eta}= |\{t=1,\dots, k|l_{\eta(t)}=l_{\eta(k)} \mbox{ and } i_{\eta(t)}>i_{\eta(k)} \}|.\]

\begin{lm}\label{surj2}
Assume that $\lambda$ is $(I|J)$-robust. Let $\eta\in \Sigma_k$ and $L\in S(\lambda,I,J)$. 
Then the matrix of the map $\psi_k:S_{\eta}\to \overline{S}_{I|J}$ for the bases $\mathfrak{B}_\eta$ ordered by $<^\eta$, and $\mathfrak{B}$ ordered by $<$, is an upper-triangular matrix. Its diagonal entry
corresponding to $\pi_{\eta I_0|\eta L}$ and $\overline{\pi}_{I_0|L}$ is
\[\alpha_{\eta I_0|\eta L}= (-1)^\eta (\omega_{i_{\eta(k)}, l_{\eta(k)}}-a_{L,\eta}+b_{L,\eta}).\]
\end{lm}
\begin{proof}
By Proposition \ref{tool}, we have 
\[\begin{aligned}\psi_k(\pi_{\eta I_0|\eta L})= &\omega_{i_{\eta(k)}l_{\eta(k)}} \overline{\pi}_{\eta I_0|\eta L} \\
&+ (1-\delta_{i_{\eta(k)}>i_{\eta(1)}} -\delta_{l_{\eta(k)}<l_{\eta(1)}}) \overline{\pi}_{\eta I_0|l_{\eta(k)}l_{\eta(2)}\cdots l_{\eta(k-1)}l_{\eta(1)}}\\
&+ (1-\delta_{i_{\eta(k)}>i_{\eta(2)}} -\delta_{l_{\eta(k)}<l_{\eta(2)}}) \overline{\pi}_{\eta I_0|l_{\eta(1)}l_{\eta(k)}l_{\eta(3)}\cdots l_{\eta(k-1)}l_{\eta(2)}} + \ldots \\
&+ (1-\delta_{i_{\eta(k)}>i_{\eta(k-1)}} -\delta_{l_{\eta(k)}<l_{\eta(k-1)}}) \overline{\pi}_{\eta I_0|l_{\eta(1)}\cdots l_{\eta(k)}l_{\eta(k-1)}}.\end{aligned}\]

Consider the general term 
\[ (1-\delta_{i_{\eta(k)}>i_{\eta(t)}} -\delta_{l_{\eta(k)}<l_{\eta(t)}}) 
\overline{\pi}_{\eta I_0|l_{\eta(1)}\cdots l_{\eta(t-1)} l_{\eta(k)} l_{\eta(t+1)}\cdots l_{\eta(k-1)}l_{\eta(t)}}\]
in the above formula for $\psi_k(\pi_{\eta I_0|\eta L})$.

If $i_{\eta(k)}>i_{\eta(t)}$ and $l_{\eta(k)}\geq l_{\eta(t)}$,  or $i_{\eta(k)}\leq i_{\eta(t)}$ and $l_{\eta(k)}<l_{\eta(t)}$, then this term vanishes.

If $i_{\eta(t)}=i_{\eta(k)}$ and $l_{\eta(t)}<l_{\eta(k)}$, then 
\[\overline{\pi}_{\eta I_0|l_{\eta(1)}\cdots l_{\eta(t-1)} l_{\eta(k)} l_{\eta(t+1)}\cdots l_{\eta(k-1)}l_{\eta(t)}}=
-\overline{\pi}_{\eta I_0|\eta L}.\]

If $i_{\eta(t)}>i_{\eta(k)}$ and $l_{\eta(t)}=l_{\eta(k)}$, then 
\[\overline{\pi}_{\eta I_0|l_{\eta(1)}\cdots l_{\eta(t-1)} l_{\eta(k)} l_{\eta(t+1)}\cdots l_{\eta(k-1)}l_{\eta(t)}}=
+\overline{\pi}_{\eta I_0|\eta L}.\]

Adding up $\omega_{i_{\eta(k)}l_{\eta(k)}} \overline{\pi}_{\eta I_0|\eta L}$, and all other terms considered so far, we obtain
\[(\omega_{i_{\eta(k)}, l_{\eta(k)}}-a_{L,\eta}+b_{L,\eta}) \overline{\pi}_{\eta I_0|\eta L}.\] 
Since $i_{\eta(t)}=i_{\eta(k)}$ and $l_{\eta(t)}=l_{\eta(k)}$ is not possible, it remains to analyze two cases: 
$i_{\eta(k)}>i_{\eta(t)}$ and $l_{\eta(k)}<l_{\eta(t)}$; and $i_{\eta(k)}<i_{\eta(t)}$ and $l_{\eta(k)}>l_{\eta(t)}.$

Let $\gamma=(\eta(t)\eta(k))$ be the transposition switching entries in positions $\eta(t)$ and $\eta(k)$ and $K=\gamma.L$. 
Then the above general term becomes
\[ (1-\delta_{i_{\eta(k)}>i_{\eta(t)}} -\delta_{l_{\eta(k)}<l_{\eta(t)}}) \overline{\pi}_{\eta I_0|\eta K}.\]

If $i_{\eta(k)}>i_{\eta(t)}$ and $l_{\eta(k)}<l_{\eta(t)}$, then $\eta(t)<\eta(k)$ and $l_{\eta(t)}>l_{\eta(k)}$; applying $\gamma$ brings higher value earlier and smaller value later,
which imply that $K=\gamma.L\lneq L$. Additionally, since $\gamma$ exchanges positions corresponding to different values of $i$, 
we obtain that either $\overline{\pi}_{I_0|K}=0$, or otherwise it equals to an element from $\mathfrak{B}$ that is smaller than $\overline{\pi}_{I_0|L}$.
If $i_{\eta(k)}<i_{\eta(t)}$ and $l_{\eta(k)}>l_{\eta(t)}$, then $\eta(t)>\eta(k)$ and $l_{\eta(t)}<l_{\eta(k)}$ imply again that $K=\gamma.L\lneq L$ 
and, analogously as above, we either get $\overline{\pi}_{I_0|K}=0$ or it equals to an element from $\mathfrak{B}$ that is smaller than $\overline{\pi}_{I_0|L}$.

Therefore, nonzero coefficients in the expression for $\psi_k(\pi_{\eta I_0|\eta L})$
as a linear combination of elements from $\mathfrak{B}$ occur only at $\overline{\pi}_{I_0|M}$, where $M\leq L$ and the statement of the Lemma follows.
\end{proof}

\begin{pr}\label{rob-link2}
Assume that $\lambda$ is $(I|J)$-robust. 
If the simple $G$-supermodule $L_G(\lambda_{I|J})$ is a composition factor of $H^0_G(\lambda)$, then $\lambda \sim_{odd} \lambda_{I|J}$.

If $\lambda$ and $\lambda_{I|J}$ are both polynomial weights and the simple $G$-supermodule $L_G(\lambda_{I|J})$ is a composition factor of $\nabla(\lambda)$, then $\lambda \sim_{podd} \lambda_{I|J}$.
\end{pr}
\begin{proof}
We use Proposition \ref{comp}, and assume that the map $\psi_k:S_{I|J} \to \overline{S}_{I|J}$ is not surjective.

We create collections $C_N$ indexed by multi-indices $N=(n_1, \ldots, n_k)$ of content $\cont(J)$ such that $(I_0|N)$ is admissible.
The collection $C_N$ consists of vectors $\pi_{\eta I_0|\eta N}$ for all $\eta\in \Sigma_k$. According to Lemma \ref{surj2}, the images 
$\psi_k(\pi_{\eta I_0|\eta N})$ of elements in the collection $C_N$ expressed as a linear combination of vectors $\overline{\pi}_{I_0|L}$ have the property that 
the coefficient at $\overline{\pi}_{I_0|N}$ equals $\alpha_{\eta I_0|\eta N}$, and all other nonzero coefficients appear only at $\overline{\pi}_{I_0|L}$ for $L<N$.
To each collection $C_N$, we assign a set $O_N=\{\alpha_{\eta I_0|\eta N}|\eta\in \Sigma_k\}$.

It is a crucial observation that if every $O_N$ contains a nonzero element, then the map $\psi_k: S_{I|J} \to \overline{S}_{I|J}$ is surjective. 
This is because by Lemma \ref{surj2}, we can find a set of vectors $\pi_{K|L}$ of weight $\lambda_{I|J}$ such that the matrix of $\psi_k$ restricted on the span of these vectors is invertible.

Therefore, if $\psi_k: S_{I|J} \to \overline{S}_{I|J}$ is not surjective, then there is $N$ as above such that all elements of $O_N$ are equal to zero. 

For $t=1, \ldots, k$ denote $\kappa_t=\lambda_{i_k\cdots i_t|n_k\cdots n_t}$, and set $\kappa_{k+1}=\lambda$.
Fix $1\leq t \leq k$ and choose a permutation $\eta\in \Sigma_k$ such that $\eta(k)=t$. Since $(I_0|N)$ is admissible, 
we obtain $(\kappa_t)^+_{i_t}=\lambda^+_{i_t}-a_{N,\eta}$ and $(\kappa_t)^-_{n_t}=\lambda^-_{n_t}+b_{N,\eta}$, 
which implies  
\[\omega_{i_t,n_t}(\kappa_t)=\omega_{i_t,n_t}-a_{N,\eta}+b_{N,\eta}=(-1)^\eta\alpha_{\eta I_0,\eta N}=0\]
and $\kappa_{t+1}\sim_{sodd} \kappa_t$.
Therefore  $\lambda$ and $\lambda_{I|J}$ are odd-linked via the sequence
\[\lambda \sim_{sodd} \lambda_{i_k|n_k} \sim_{sodd} \lambda_{i_ki_{k-1}|n_kn_{k-1}} \sim_{sodd} \cdots \sim_{sodd} 
\lambda_{i_k\cdots i_1|n_k\cdots n_1}=\lambda_{I|J}.\]

The last statement now follows from Corollary 7.2 of \cite{fm-p2}.
\end{proof}

\section{Odd linkage for polynomial weights}\label{sec5}

The category of polynomial $G=GL(m|n)$-supermodules of degree $r$ is equivalent to the category of  supermodules over the Schur superalgebra $S_r=S(m|n,r)$.
In particular, if the length of the polynomial weight $\lambda$ is $r$, then the largest polynomial subsupermodule 
$\nabla(\lambda)$ of $H^0_G(\lambda)$ corresponds to the costandard supermodule $\nabla_{S_r}(\lambda)$ of $S_r$.

A weight $\lambda$ is called polynomial if $\nabla(\lambda)\neq 0$. It was established in \cite{br} that polynomial weights $\lambda$ of $GL(m|n)$ are indexed by $(m|n)$-hook partitions. A partition 
$\lambda=(\lambda_1, \ldots , \lambda_s)$ is a \emph {$(m|n)$-hook partition} if $\lambda_{m+1}\leq n$.

Throughout this section, we assume that $\lambda=(\lambda_1, \ldots , \lambda_s)$ is an $(m|n)$-hook partition of fixed length $r$. Such $\lambda$ correspond uniquely to a dominant and polynomial weight $(\lambda^+|\lambda^-)$ of $G$, where 
$\lambda^+=(\lambda_1,\ldots, \lambda_m)$ and $\lambda^-$ is the conjugate partition to 
$(\lambda_{m+1}, \ldots, \lambda_s)$. We use both notations interchangeably.
We also consider dominant weights $\mu$ of the form $\mu=\lambda_{I|J}$ for admissible $(I|J)$. It is easy to see that such weights $\mu$ correspond to $(m|n)$-hook partitions, and they are of the same length $r$ as $\lambda$.

In this section, we investigate the simple composition factors of the costandard module $\nabla(\lambda)$ for polynomial weight $\lambda$, or equivalently, the simple composition factor of the costandard module $\nabla_{S_r}(\lambda)$ 
for $(m|n)$-hook partition $\lambda$.

Since we could work over the Schur superalgebra $S_r$, we can apply combinatorial techniques involving partitions and tableaux.

\subsection{Review of concepts and notations from \cite{fm-p2}}
\subsubsection{Skew partitions and tableaux}

We review the most relevant concepts from \cite{fm-p2}. For further information, the reader is referred to the article \cite{fm-p2}.

The partial order $\leq$ on partitions is defined by $\beta\leq \alpha$ if and only if $\beta_i\leq \alpha_i$ for every component $i$. 

Let $\lambda=(\lambda^+|\lambda^-)$ be an $(m|n)$-hook partition, and a partition $\mu$ is such that $\mu<\lambda^+$. Denote by $T$ a skew tableau of shape $\lambda'/ \mu'$ satisfying the following conditions:

1) Its entries are from the set $m+1, \ldots, m+n$ such that its content is $(0|\nu)$, where $\nu$ is a partition.

2) For each $j>m$ and $1\leq i\leq n$, the entry at the position $[ij]$ of the tableau $T$ is $m+i$. 

These assumptions imply that $\lambda^-<\nu$. 
Tableau $T$, as above, consists of two parts. The first part $T^+$, is a skew tableau of the shape $(\lambda^+/ \mu)'$ and the content $(0|\nu / \lambda^-)$.
The second part is the canonical row tableau $T^-_{can}$ corresponding to the diagram $[\lambda^-]$, where its  $i$-th row is filled with entries $m+i$ for each $i$. 

Denote by $T^{opp}$ a tableau of the shape $\nu$ and the content $(\lambda^+ /\mu|\lambda^-)$ defined as follows:

1) For each $1\leq i\leq n$ and $1\leq j\leq \lambda^-_i$, the entry at the position $[ij]$ of the tableau $T^{opp}$ is equal to $m+i$. 

2) The remaining entries are from the set $1, \ldots, m$.

The tableau $T^{opp}$ consists of two parts. The first part is the canonical row tableau $T^-_{can}$ of shape $\lambda^-$. The second part is a skew tableau $T^-$ of the shape $\nu/ \lambda^-$ and the content
$(\lambda^+ /\mu|0)$. 

Denote by $\mathcal{D}^+$ the diagram $[\lambda^{+'}/ \mu']$ corresponding to $T^+$ and by $\mathcal{D}^-$ the diagram $[\nu/ \lambda^-]$ corresponding to $T^-$.

To a skew tableau $Q$ we assign a word $w = w(Q)$ that is given by reading and concatenating entries in its rows from right to left starting in the top row and proceeding to the bottom row.

Define the words $w=w(T^+)$ and $z=w(T^+_{can})$, where $T^+_{can}$ is the canonical tableau of the same skew shape as $T^+$ such that its $j$th column is filled with entries $j$ for each $j$. 
We will use the following correspondence $Rp$ between tableaux $T^+$ and $Rp(T^+)=T^-$:
If the $t$th symbol $w_t=m+i$ appears in $w$ for the $j$th time, then 
$t^-_{i, \lambda^-_i+j}=z_t$. If the symbol $w_t$ as above appears at the position $[r_t,s_t]$ in $\mathcal{D}^+$, 
then we define a bijective map $Rpos:\mathcal{D}^+\to \mathcal{D}^-$ such that 
$Rpos([r_t,s_t])=[i,\lambda^-_i+j]$ for each $t=1, \ldots, r$. Corresponding to this, each permutation $\sigma$ of diagram $\mathcal{D}^-$ induces a permutation $Rp(\sigma)$ of $\mathcal{D}^+$.

\subsubsection{Operator $\sigma$.}\label{tensorsetup}
Assume $(I|J)$ is of length $r$, $cont(I)=(\lambda^+)'-\mu'$ and $\cont(J)=\nu-\lambda^-$.
Then there is a bijective positioning map $P^+:\{1, \ldots, r\} \to \mathcal{D}^+$ which satisfies the property that 
$P^+(l)=[s,r]$ implies $r=i_l$. 
To $(I|J)$ we asssign the tableau $T^+$ of shape $(\lambda^+\setminus \mu)'$  and content $(0,\nu\setminus \lambda^-)$  via $T^+[P^+(l)]=j_l$ for each $l=1, \ldots, r$.

Let $X^+$ be the subgroup of the symmetric group $\Sigma_r$ consisting of row permutations of $\mathcal{D}^+$.
For $\sigma\in X^+$ denote by $\sigma(T^+_{can})$, the tableau obtained by applying permutation $\sigma$ to the entries of $T^+_{can}$.
The action of $\sigma\in X^+$ on $(I|J)$ is given as $\sigma.(I|J)=(K|J)$, where for each $1\leq l\leq k$ the index $k_l$ is the entry at the position $P^{+}(l)$ 
 in $\sigma(T^+_{can})$. The operator $\sigma^+$ is given as 
\[\sigma^+=\sum_{\sigma\in X^+} (-1)^\sigma \sigma.\]

Analogously, there is a bijective positioning map $P^-:\{1, \ldots, r\}\to \mathcal{D}^-$ satisfying the property that 
$P^-(l)=[r,s]$ implies $m+r=j_l$.
To $(I|J)$ we assign the tableau $T^-$ of shape $\nu\setminus \lambda^-$ and content $(\lambda^+\setminus \mu)'$   via $T^-[P^+(l)]=j_l$ for each $l=1, \ldots, r$.
Let $X^-$ be the subgroup of the symmetric group $\Sigma_r$ consisting of column permutations of $\mathcal{D}^-$.
For $\sigma\in X^-$ denote by $\sigma(T^-_{can})$, the tableau obtained by applying permutation $\sigma$ to the entries of $T^-_{can}$. 
The action of  $\sigma\in X^-$ on $(I|J)$ is given as $\sigma.(I|J)=(I|L)$, where for each $1\leq a\leq k$, the index $l_a$ is the entry at the position 
$P^{-}(a)$ in $\sigma(T^-_{can})$.
The operator $\sigma^-$ is given as 
\[\sigma^-=\sum_{\sigma\in X^-} (-1)^\sigma \sigma.\]

Finally, the complete operator $\sigma$ is defined as a composition $\sigma^-\circ \sigma^+$.

This action of $\sigma$ on multi-indices $(I|J)$ extends to the action on bideterminants - see pp. 390-391 of \cite{fm-p2} - and on $\rho_{I|J}$. One of  the main results of \cite{fm-p2}, Theorem 4.3, implies that elements $\pi_{I|J}=v_{I|J}\sigma\rho_{I|J}$ belong to $S_{I|J}$.

\subsubsection{Operator $\tau$}
To a tableau $Q^+$ of shape $(\lambda^+\setminus \mu)'$ and content $(0| \nu\setminus \lambda^-)$, one can find an element $(I|J)$ such that the tableau $T^+$ we assigned to $(I|J)$ earlier equals $Q^+$. Depending on the choice of the positioning map $P^+$, there are different elements $(K|L)$ that produce the same tableau $Q^+$.  However, in each case, $\overline{\rho}_{K|L}=\pm \overline{\rho}_{I|J}$.

Analogously, to a tableau $Q^-$ of shape $\nu\setminus \lambda^-$ and content $(\lambda^+\setminus \mu|0)$, 
one can find an element $(I|J)$ such that the tableau $T^-$ assigned to $(I|J)$ earlier equals $Q^-$. Depending on the  positioning map $P^-$, there are different elements $(K|L)$ that produce the same tableau $Q^-$.  
However, in each case $\overline{\rho}_{K|L}=\pm \overline{\rho}_{I|J}$. 

Therefore, when we define $\tau^+(Q^+)=\sum_{\sigma\in X^+} \sigma Q^+$ and 
$\tau^-(Q^-)=\sum_{\sigma\in X^-} \sigma Q^-$, 
the expressions
$\overline{\rho}(\tau^+ Q^+)=\sum_{\sigma\in X^+} \overline{\rho}(\sigma Q^+)$
and 
$\overline{\rho}(\tau^- Q^-)=\sum_{\sigma\in X^-} \overline{\rho}(\sigma Q^-)$ are well-defined up to a constant $\pm 1$. The operators $\overline{\rho}\tau^+$ and $\overline{\rho}\tau^-$ are viewed as a tableau analog of operators $\rho\sigma^+$ and $\rho\sigma^-$.

The biggest challenge is the combination of operators $\overline{\rho}\tau^+$ and $\overline{\rho}\tau^-$ for a single tableau $Q^+$. We have already assigned to $Q^+$ the tableau $Q^-=Rp(Q^+)$ earlier, and for $\sigma\in X^-$ we define $\sigma Q^+=Q^+\circ Rp(\sigma)$. This extends naturally to the definition of the operator $\tau$ such that 
$\overline{\rho}(\tau Q^+)= \overline{\rho}(\tau^+\tau^- Q^+)$. 

\subsubsection{Vectors $\overline{v}(Q^+)$.}

Theorem 5.2 of \cite{fm-p2} asserts that for a tableau $Q^+$ of shape $(\lambda^+\setminus \mu)'$ and content 
$(0|\nu\setminus \lambda^-)$ the expression $\overline{v}(Q^+)=v_{I|J}\overline{\rho}(\tau Q^+)$ is a $G_{ev}$-primitive vector of $H^0_G(\lambda)\otimes \wedge(Y)$; hence it belongs to $\overline{S}_{I|J}$.

It is essential that the specific repositioning map $Rp$, which sends $Q^+$ to $Q^-$, which we have chosen earlier, behaves well with respect to Yamanouchi words and Littlewood-Richardson tableaux - See section 5.3 and 6.2 of \cite{fm-p2}.

Recall that a tableau $Q$ is called {\em semistandard} if its entries in each row are weakly increasing from left
to right, and all entries in each column are strictly increasing from top to bottom.

We call $Q$ {\em anti-semistandard} if the entries in its rows are strictly decreasing from left to right, and entries in its columns are weakly decreasing from top to bottom.

If $Q^+$ is semistandard tableau  and the corresponding tableau $Q^-$ is anti-semistandard,
then $Q^+$ is called {\em marked}. The set of all marked tableaux $Q^+$ of the shape $(\lambda^+\setminus \mu)'$ 
and the content $(0|\nu \setminus \lambda^-)$ is denoted by $M(((\lambda^+)'\setminus \mu',\nu\setminus \lambda^-)$.

\subsection{Even-primitive vectors}

In the general case, we still work with the vectors $\pi_{I|J}$ and rely on the formulas for $\psi_k(\pi_{I|J})$ derived in the previous sections. 
The following approach works for arbitrary $\lambda_{I|J}$ and essentially reduces the general case to the case when $\lambda$ is $(I|J)$-robust.

Since $\pi_{I|J}$ do not necessarily belong to $T_k$, we use 
Theorem 4.3 of \cite{fm-p2} to obtain suitable even-primitive vectors of $T_k$ that are integral linear combinations of vectors $\pi_{K|L}$. 
These vectors belong to $S_{I|J}$ and are of the form $v_{I|J}\sigma.\rho_{I|J}$

On the other hand, by Theorem 7.1 of \cite{fm-p2} (see also Theorem 6.24 of \cite{fm-p2}), the set 
$\mathfrak{B}=\{\overline{v}(T^+)|T^+\in M({\lambda^+}'/\mu',\nu/\omega)\}$ form a basis of the space $B$ of all even-primitive vectors in 
$\nabla_k(\lambda)=\nabla(\lambda)\cap F_k$ of weight $(\mu|\nu)$.

Let us fix an arbitrary reading of all tableaux $T^+$ of the skew shape ${\lambda^+}'/\mu'$. Corresponding to such a reading, we assign two words. The word $J_{T^+}$ is obtained by listing the entries $j_t$ in the same order as the symbols $m+j_t$ are read. If the reading of the $t$th symbol of $T^+$ appears at the location $(k_t,i_t)\in \mathcal{D}^+$, then we define
$I_0=(i_1, \ldots, i_k)$. Note that $I_0$ is the same for all tableaux $T^+$ since it only depends on the fixed reading, but the entries of $I_0$ are no longer nondecreasing 
(as was the case earlier). 

List tableux $T^+\in M({\lambda^+}'/\mu',\nu/\omega)$ in a certain order $T^+_1<\ldots <T^+_l$. Corresponding to 
this, we list the elements of $\mathfrak{B}$ as $\mathfrak{B}=\{\overline{v}(T^+_1), \ldots, \overline{v}(T^+_l)\}$ 
and write $J^s=J_{T^+_s}=(j^s_1, \ldots, j^s_k)$  for $s=1, \ldots, l$.

For a permutation $\eta\in \Sigma_k$, define $\mathfrak{B}_\eta=\{\sigma\pi_{\eta I_0|\eta J^s}|s=1, \ldots, l\}$ and $B_\eta=span(\mathfrak{B}_\eta)$.
The image of the element $\sigma\pi_{\eta I_0|\eta J^s}$ from $\mathfrak{B}_\eta$ under the map $\proj: K(m|n)\otimes Y^{\otimes k} \to K(m|n)\otimes \wedge^k Y$ 
is $\overline{v}(T^+_s)=\sigma \overline{\pi}_{I_0|J^s}$.
Since $\psi_k$ is the $G_{ev}$-morphism by Lemma \ref{evmorph}, the restriction of $\psi_k$ to $B_\eta$ gives a $G_{ev}$-morphism from $B_\eta$ to $B$.
Since $\mathfrak{B}_\eta$ is a basis of $B_\eta$ and $\mathfrak{B}$ is a basis of $B$, we can consider the matrix $M_\eta$ of $\psi_k$ for the bases
$\mathfrak{B}_\eta$ and $\mathfrak{B}$.

For $T^+\in M({\lambda^+}'/\mu',\nu/\omega)$ we define the set $\mathfrak{C}_{T^+}=\{\sigma\pi_{\eta I_0|\eta J_{T^+}}| \eta\in \Sigma_k\}$ and denote by 
$C_{T^+}$ the $F$-span of elements in $\mathfrak{C}_{T^+}$.

\subsection{Case of $\lambda_{I|J}$ with distinct entries in $I$ and distinct entries in $J$}

Before dealing with the general case, we describe the case when all entries in $I$ are distinct, and all entries in $J$ are distinct. In this case, the arguments are much simpler, and we can avoid the full machinery of marked tableaux, pictures, and their properties.

Marked tableaux are easy to understand in this case. A tableau $T^+$ belongs to $M({\lambda^+}'/\mu',\nu/\omega)$ if and only if entries in all rows of $T^+$ are strictly increasing and entries in all columns of $T^-$ are 
strictly decreasing. 

Denote by $\overline{T^+}$ the row tabloid corresponding to $T^+$.

\begin{lm}\label{disjoint}
Every $\overline{v}(T^+)\in \mathfrak{B}$, where $T^+\in M({\lambda^+}'/\mu',\nu/\omega)$, is an integral linear combination of various $\overline{\pi}_{I_0|L}$, where $\cont(L)=\cont(J)$. 
Each term $\overline{\pi}_{I_0|L}$ can appear with a non-zero coefficient in at most one $\overline{v}(T^+)\in \mathfrak{B}$.
\end{lm}
\begin{proof}
The first statement follows from Theorem 7.1 of \cite{fm-p2}. 
For the second part, first observe that since every $T^+\in M({\lambda^+}'/\mu',\nu/\omega)$ is semistandard, it is uniquely described by the content vectors of each row 
in the tableau $T^+$, or equivalently by its row tabloid $\overline{T^+}$. 

If $T^+\in M({\lambda^+}'/\mu',\nu/\omega)$, then every tableau appearing in $\tau^- T$ is obtained by permuting entries of $T^+$ that correspond to entries from the same columns of $T^-$ 
(and same columns of the diagram $[\nu/ \omega]$).

Additionally, entries from the same column of $T^-$ correspond to entries in $T^+$ that belong to different rows of $T^+$ (and different rows of the skew diagram
$[{\lambda^+}'/ \mu']$.)  

Assume that $T^+_1, T^+_2 \in M({\lambda^+}'/\mu',\nu/\omega)$ are such that there is a common summand in $\overline{v}(T^+_1)$ and $\overline{v}(T^+_2)$. Then there is a common summand in 
$\overline{\rho}(\tau T^+_1)$ and $\overline{\rho}(\tau T^+_2)$ and, consequently, a common summand in $\tau T^+_1$ and $\tau T^+_2$.
In this case, there is a summand $R_1$ in $\tau^-T^+_1$ and a summand $R_2$ in $\tau^-T^+_2$ such that the row tabloids $\overline{R_1}$ and $\overline{R_2}$ coincide. Let $R_1=\sigma_{-,1}T^+_1$ and $R_2=\sigma_{-,2}T^+_2$, where $\sigma_{-,1}, \sigma_{-,2}$ are column permutations of $[\nu/\omega]$.
Then every entry $m+j_l$ appears in the same row in $R_1$ as in $R_2$. Since entries in $J$ are distinct, and entries in columns of $T^-_1, T^-_2$ are strictly decreasing, we conclude that $\sigma_{-,1}=\sigma_{-,2}$ and $\overline{T^+_1}=\overline{T^+_2}$. Since both $T^+_1, T^+_2$ are semistandard, this implies $T^+_1=T^+_2$.
\end{proof}

To each set $\mathfrak{C}_s=\mathfrak{C}_{T^+_s}$, we assign a set $O_s=O_{T^+_s}=\{\omega_{i_{\eta(k)},j^s_{\eta(k)}}|\eta\in\Sigma_k\}= \{\omega_{i_t,j^s_t}|t=1, \ldots k\}$.

The following example, related to Example \ref{ex1}, illustrates this setup.
To compute the entries in the appearing matrices, we use Proposition \ref{tool}.

\begin{ex}\label{ex4}
Let $m=n=3$ and a weight $\lambda$ be such that $\lambda^+_1=\lambda^+_2>\lambda^+_3$ and $\lambda^-_1>\lambda^-_2>\lambda^-_3$. Assume that $I=J=(123)$. 
Then in the notation of Example \ref{ex1}, 
\[\mathfrak{B}=\{\overline{v}(T^+_1)=\vec{b}_1+\vec{b}_2,\overline{v}(T^+_2)=\vec{b}_3+\vec{b}_4,\overline{v}(T^+_3)=\vec{b}_5+\vec{b}_6\}\]
is a basis of all even-primitive vectors of weight $\lambda_{I|J}$ in $\nabla_k(\lambda)$.

The sets $\mathfrak{C}_s$ are given as follows.
\[\begin{aligned}&\mathfrak{C}_{1}=\{\vec{a}_1-\vec{a}_8, \vec{a}_7-\vec{a}_2, \vec{a}_{13}-\vec{a}_{26}, \vec{a}_{19}-\vec{a}_{32}, \vec{a}_{25}-\vec{a}_{14}, \vec{a}_{31}-\vec{a}_{20}\}, \\
&\mathfrak{C}_{2}=\{\vec{a}_3-\vec{a}_{10}, \vec{a}_9-\vec{a}_4, \vec{a}_{15}-\vec{a}_{28}, \vec{a}_{21}-\vec{a}_{34}, \vec{a}_{27}-\vec{a}_{16}, \vec{a}_{33}-\vec{a}_{22}\}, \\
&\mathfrak{C}_{3}=\{\vec{a}_5-\vec{a}_{12}, \vec{a}_{11}-\vec{a}_6, \vec{a}_{17}-\vec{a}_{30}, \vec{a}_{23}-\vec{a}_{36}, \vec{a}_{29}-\vec{a}_{18}, \vec{a}_{35}-\vec{a}_{24}\}.
\end{aligned}\] 
The assumption $\lambda^+_1=\lambda^+_2$ implies $\omega_{1j}=\omega_{2j}+1$. Using this, we evaluate the following matrices. 

The matrix of $\psi_k$ for the bases $\mathfrak{C}_1$ and $\mathfrak{B}$ is given as
\[\begin{pmatrix}
\omega_{33}&-\omega_{33}&-\omega_{12}&\omega_{12}&\omega_{21}&-\omega_{21}\\
0&0&0&0&0&0\\
0&0&0&0&0&0
\end{pmatrix},\]
the matrix of $\psi_k$ for the bases $\mathfrak{C}_2$ and $\mathfrak{B}$ is given as
\[\begin{pmatrix} 
-1&1&-1&1&0&0\\
\omega_{32}&-\omega_{32}&-\omega_{13}&\omega_{13}&\omega_{21}&-\omega_{21}\\
0&0&0&0&0&0
\end{pmatrix}\]
and the matrix of $\psi_k$ for the bases $\mathfrak{C}_3$ and $\mathfrak{B}$ is given as
\[\begin{pmatrix}
-1&1&-1&1&0&0\\
-1&1&0&0&1&-1\\
\omega_{31}&-\omega_{31}&-\omega_{13}&\omega_{13}&\omega_{22}&-\omega_{22}
\end{pmatrix}.\]

Therefore we have sets
\[O_1=\{\omega_{33}, \omega_{12}, \omega_{21}\}, 
O_2=\{\omega_{32}, \omega_{13}, \omega_{21}\}, 
O_3=\{\omega_{31}, \omega_{13}, \omega_{22}\}\]
and if $\psi_k$ is not surjective, then at least one of $O_1, O_2, O_3$ contains only zeros and therefore
$\lambda\sim_{odd}\lambda_{I|J}$.
\end{ex}

It is helpful to keep in mind the above example in the following general consideration.

Next, we replace an arbitrary order on marked tableaux by a specific one. 
The reverse Semitic lexicographic order $<$ on multi-indices $L$ of the same content as $J$ extends to a linear order on tableaux $T^+$ and on vectors $\overline{\pi}_{I_0|J_{T^+}}$. 
List the elements of $M({\lambda^+}'/\mu',\nu/\omega)$ with respect to this order $<$ as
$T^+_1< T^+_2 <\ldots < T^+_l$. 

Then we have 
\[\overline{\pi}_{I_0|J^1}=\mathfrak{l}(T^+_1)< \overline{\pi}_{I_0|J^2}=\mathfrak{l}(T^+_2) < \ldots < \overline{\pi}_{I_0|J^l}=\mathfrak{l}(T^+_l).\]
This induces the order $<^\eta$ on elements of $\mathfrak{B}_\eta$ such that $\sigma\pi_{\eta I_0|\eta J_{T^+_a}}<^\eta \sigma\pi_{\eta I_0|\eta J_{T^+_b}}$ if and only if $a<b$.

If $T^+\in M({\lambda^+}'/\mu',\nu/\omega)$, then the summands in $\overline{v}(T^+)$ with respect to $<$ are such that $\mathfrak{l}(T^+)=\overline{\pi}_{I_0|J_{T^+}}$ is the highest term in $\overline{v}(T^+)$. Therefore we call $\overline{\pi}_{I_0|J_{T^+}}$ the leading element of $\overline{v}(T^+)$.

\begin{lm}\label{surj3}
Assume that all entries in $I$ are distinct, all entries in $J$ are distinct, $\eta\in\Sigma_k$ and $T^+\in M({\lambda^+}'/\mu',\nu/\omega)$. 
The matrix of the map $\psi_k:B_\eta\to B$ for the bases $\mathfrak{B}_\eta$ ordered by $<^\eta$, and $\mathfrak{B}$ ordered by $<$, is an upper-triangular matrix. 
Its diagonal entry corresponding to 
$\sigma\pi_{\eta I_0|\eta J_{T^+_s}}$ and $\overline{v}(T^+_s)$ is $(-1)^\eta\omega_{i_{\eta(k)}, j^s_{\eta(k)}}$.
\end{lm}
\begin{proof}
If $\pi_{K|M}$ is a summand of an element $\sigma \pi_{\eta I_0|\eta J_{T^+}}$ from $\mathfrak{B}_\eta$, then 
$\overline{\pi}_{K|M}\leq \overline{\pi}_{I_0,J_{T^+}}= \mathfrak{l}(T^+)$.

Using Lemma \ref{surj}, we derive that if the coefficient at $\overline{\pi}_{I_0|N}$ in the linear combination expressing $\psi_k(\pi_{K|M})$ is nonzero, 
then $N\leq J_{T^+}$. Therefore, if the coefficient at $\overline{\pi}_{I_0|N}$ in the linear combination expressing $\psi_k(\sigma \pi_{\eta I_0|\eta J_{T^+}})$ is nonzero, 
then $N\leq J_{T^+}$.

Moreover, the expression $\pi_{\eta I_0|\eta J_{T^+}}$ is the only summand in $\sigma\pi_{\eta I_0|\eta J_{T^+}}$ such that its image under $\psi_k$ has a nonzero coefficient, 
namely $(-1)^\eta$, at $\overline{\pi}_{I_0|J_{T^+}}$.

Therefore, Lemmas \ref{disjoint} and \ref{surj} imply that if we express $\psi_k(\sigma\pi_{\eta I_0|\eta J_{T^+_s}})$ as a linear combination of elements $\overline{v}(T^+_t)$ for 
$T^+_t\in M({\lambda^+}'/\mu',\nu/\omega)$, then the coefficient at $\overline{v}(T^+_s)$ is $(-1)^\eta \omega_{i_{\eta(k)},j^s_{\eta(k)}}$ and 
all coefficients at $\overline{v}(T^+_t)$ for $t>s$ vanish.
\end{proof}

\begin{teo}\label{linkdist}
Assume that all entries in $I$ are distinct, and all entries in $J$ are distinct. Moreover, assume that $\lambda$ and $\lambda_{I|J}$ are dominant and polynomial weights.
If the simple $G$-supermodule $L_G(\lambda_{I|J})$ is a composition factor of $\nabla(\lambda)$, then $\lambda \sim_{podd} \lambda_{I|J}$. 
\end{teo}
\begin{proof}
We use Corollary 7.2 of \cite{fm-p2} and Proposition \ref{comp}, and assume that the map $\psi_k:S_{I|J} \to \overline{S}_{I|J}$ is not surjective. 

It is a crucial observation that if every $O_{T^+_s}$ contains a nonzero element, then the map $\psi_k: S_{I|J} \to \overline{S}_{I|J}$ is surjective. This is because 
using Lemma \ref{disjoint} and Theorem 7.1 of \cite{fm-p2}, we can find a set of primitive vectors $\sigma\pi_{\eta_s I_0|\eta_s J_{T^+_s}}$ from  $S_{I|J}$ such that the matrix of $\psi_k$ restricted on the span of these vectors is invertible.

Therefore, if $\psi_k: S_{I|J} \to \overline{S}_{I|J}$ is not surjective, then there is $s$ such that all elements of $O_{T^+_s}$ are equal to zero. 
In this case, we derive that $\lambda$ and $\lambda_{I|J}$ are odd-linked via a sequence 
$\lambda\sim_{sodd}\lambda_{i_1|j^s_1} \sim_{sodd} \lambda_{i_1i_2|j^s_1j^s_2} \sim_{sodd} \cdots \sim_{sodd} \lambda_{I|J_s}=\lambda_{I|J}$ because 
$\omega_{i_1,j^s_1}(\lambda)=0$, $\omega_{i_2,j^s_2}(\lambda_{i_1,j^s_1})=\omega_{i_2,j^s_2}(\lambda)=0$ and so on until
\linebreak $\omega_{i_k,j^s_k}(\lambda_{i_1i_2\cdots i_{k-1},j^s_1j^s_2\cdots j^s_{k-1}})=0$.
Thus, $\lambda\sim_{odd}\lambda_{I|J}$. Since $\lambda_{I|J}$ is a polynomial weight, all intermediate $\lambda_{i_1i_2\cdots i_t|j^s_1j^s_2\cdots j_t}$ are also polynomial weights, hence $\lambda\sim_{podd}\mu$.
\end{proof}

\subsection{General case of $\lambda_{I|J}$}

Before we make our final choices of the reading of the tableau $T^+$ and ordering $<$, let us comment on the previous particular cases where other options often seemed more natural.

If all entries in $I$ are distinct, then $T^+$ is a row strip; if all entries in $J$ are distinct, then $T^-$ is a column strip.

If the weight $\lambda$ is $(I|J)$-robust, then $T^+$ is a column strip, and $T^-$ is a row strip.

Assume $\lambda$ is $(I|J)$-robust. Then all entries in $I$ are distinct if and only if $T^+$ is an antichain with respect to the order $<_{\nwarrow}$ defined in \cite{fm-p2} 
(or \cite{van1}), and all entries in $J$ are distinct if and only if $T^-$ is an antichain with respect to the order $<_{\nwarrow}$.
The case of an antichain is the simplest, and the reading of tableaux does not matter.

Even the case when all entries in $I$ are distinct, and all entries in $J$ are distinct - which we were considering above - is unique from this point of view, and it is not surprising that we could use various readings of the tableaux.

In the most general case, which we will consider now, the tableaux could assume complex skew shapes, and we have a unique choice for our setup to work. 
Namely, from now on, we assume that the reading of a tableau $T^+$ is by rows moving from right to left and moving from top to bottom.

Previously, in the particular cases described in Lemmas \ref{surj}, \ref{surj2}, and \ref{surj3}, we have worked with the reverse Semitic order. We will show that the reverse Semitic  order in Lemma \ref{surj2} can be replaced by the Clausen row order that works in general.  

For every index $i$ corresponding to a row of $T^+$ and a number $1 \leq k \leq n$, 
the number of occurrences of symbols $\{m+1, . . . ,m+k\}$
in the rows of $T^+$ of index $i$ or lower is denoted by $r_{ik}$.
The Clausen row matrix of the tableau $T^+$ is given as $(r_{ik})$.

The Clausen row preorder on the set of set of tableaux $T^+\in M({\lambda^+}'/\mu',\nu/\omega)$ is given as follows.
Let $T^+$ and $T'^+$ are tableaux from $M({\lambda^+}'/\mu',\nu/\omega)$, $R(T^+)=(r_{ik})$, and $R(T'^+)=(r'_{ik})$. 
Then $T^+<T'^+$ if and only if $R(T^+)=R(T'^+)$ or there are indices $i$ and $k$ such that 
$r_{jl}=r'_{jl}$ for all $j<i$ and $1\leq l\leq n$, $r_{il}=r'_{il}$ for all $l<k$, and $r_{ik}<r'_{ik}$.

From now on, we fix the row Clausen preorder $<$ on tableaux $T^+\in M({\lambda^+}'/\mu',\nu/\omega)$.

\begin{lm}\label{surj4}
If $\eta\in \Sigma_k$, then the matrix of the map $\psi_k:B_{\eta}\to B$ for the bases $\mathfrak{B}_\eta$ ordered by $<^\eta$, and $\mathfrak{B}$ ordered by $<$, is an upper-triangular matrix with integral coefficients. Its diagonal entry
corresponding to $\sigma\pi_{\eta I_0|\eta L}$ and $\sigma\overline{\pi}_{I_0|L}$ is $\alpha_{\eta I_0|\eta L}$.
\end{lm}
\begin{proof}
By Proposition \ref{tool}, we have 
\[\begin{aligned}\psi_k(\pi_{\eta I_0|\eta J^s})= &\omega_{i_{\eta(k)}j^s_{\eta(k)}} \overline{\pi}_{\eta I_0|\eta J^s} \\
&+ (1-\delta_{i_{\eta(k)}>i_{\eta(1)}} -\delta_{j^s_{\eta(k)}<j^s_{\eta(1)}}) \overline{\pi}_{\eta I_0|j^s_{\eta(k)}j^s_{\eta(2)}\cdots j^s_{\eta(k-1)}j^s_{\eta(1)}}\\
&+ (1-\delta_{i_{\eta(k)}>i_{\eta(2)}} -\delta_{j^s_{\eta(k)}<j^s_{\eta(2)}}) \overline{\pi}_{\eta I_0|j^s_{\eta(1)}j^s_{\eta(k)}j^s_{\eta(3)}\cdots j^s_{\eta(k-1)}j^s_{\eta(2)}} + \ldots \\
&+ (1-\delta_{i_{\eta(k)}>i_{\eta(k-1)}} -\delta_{j^s_{\eta(k)}<j^s_{\eta(k-1)}}) \overline{\pi}_{\eta I_0|j^s_{\eta(1)}\cdots j^s_{\eta(k)}j^s_{\eta(k-1)}}.\end{aligned}\]
Therefore, the diagonal entries of our matrix are the same as described in the statement of the lemma, and it only remains to show that our matrix
is upper-triangular.

Consider the general term 
\[ (1-\delta_{i_{\eta(k)}>i_{\eta(t)}} -\delta_{j^s_{\eta(k)}<j^s_{\eta(t)}}) 
\overline{\pi}_{\eta I_0|j^s_{\eta(1)}\cdots j^s_{\eta(t-1)} j^s_{\eta(k)} j^s_{\eta(t+1)}\cdots j^s_{\eta(k-1)}j^s_{\eta(t)}}\]
in the above formula for $\psi_k(\pi_{\eta I_0|\eta J^s})$.
Let $\gamma=(\eta(t)\eta(k))$ be the transposition switching entries in positions $\eta(t)$ and $\eta(k)$ and $K=\gamma.J^s$. 

Then $\eta K=(j^s_{\eta(1)}\cdots j^s_{\eta(t-1)} j^s_{\eta(k)} j^s_{\eta(t+1)}\cdots j^s_{\eta(k-1)}j^s_{\eta(t)})$ and the above general term becomes
\[ (1-\delta_{i_{\eta(k)}>i_{\eta(t)}} -\delta_{j^s_{\eta(k)}<j^s_{\eta(t)}}) \overline{\pi}_{\eta I_0|\eta K}.\]

If $i_{\eta(k)}>i_{\eta(t)}$ and $j^s_{\eta(k)}>j^s_{\eta(t)}$,  or $i_{\eta(k)}<i_{\eta(t)}$ and $j^s_{\eta(k)}<j^s_{\eta(t)}$, then this term vanishes.

If $i_{\eta(k)}>i_{\eta(t)}$ and $j^s_{\eta(k)}<j^s_{\eta(t)}$, then $\eta(t)<\eta(k)$ and $j^s_{\eta(t)}>j^s_{\eta(k)}$.
This means that in the marked tableau $T^+_s$ we have an entry $m+j^s_{\eta(t)}$ appearing at the position $[r_t,\eta(t)]$ and 
an entry $m+j^s_{\eta(k)}$ appearing at the position $[r_k,\eta(k)]$. If $r_t\leq r_k$, then $\mathcal{D}^+$ contains the position
$[r_k,\eta(k)]$ because $\eta(t)<\eta(k)$. Since $T^+_s$ is semistandard, this implies $m+j^s_{\eta(t)}\leq m+j^s_{\eta(k)}$, which is a contradiction.
Therefore $r_t>r_k$, which means that $K=\gamma J^s < J^s$ with respect to Clausen row preorder.

If $i_{\eta(k)}<i_{\eta(t)}$ and $j^s_{\eta(k)}>j^s_{\eta(t)}$, then $\eta(t)>\eta(k)$ and $j^s_{\eta(t)}<j^s_{\eta(k)}$ imply again that $K=\gamma.J^s< J^s$. 
Therefore either $\overline{\pi}_{I_0|K}=0$ or otherwise $\overline{\pi}_{I_0|K}<\overline{\pi}_{I_0|J^s}$.

Thus nonzero coefficients in the expression for $\psi_k(\pi_{\eta I_0|\eta J^s})$
as a linear combination of elements $\overline{\pi}_{I_0|M}$ only occur when $M\leq J^s$ and the coefficient at $\overline{\pi}_{I_0|J_s}$ is $\alpha_{\eta I_0|\eta J^s}$.

Next, if $\pi_{K|M}$ is a summand of an element $\sigma \pi_{\eta I_0|\eta J^s}$ from $\mathfrak{C}_{T^+_s}$, then 
$\overline{\pi}_{K|M}\leq \overline{\pi}_{I_0,J^s}= \mathfrak{l}(T^+_s)$. 
We infer that if the coefficient at $\overline{\pi}_{I_0|N}$ in the linear combination expressing $\psi_k(\pi_{K|M})$ is nonzero, 
then $N\leq J^s$.  Therefore, if the coefficient at $\overline{\pi}_{I_0|N}$ in the linear combination expressing $\psi_k(\sigma \pi_{\eta I_0|\eta J^s})$ is nonzero, then $N\leq J^s$.

Moreover, the expression $\pi_{\eta I_0|\eta J^s}$ is the only summand $\pi_{U|V}$ in $\sigma\pi_{\eta I_0|\eta J^s}$ such that $\psi_k(\pi_{U|V})$ has a nonzero coefficient, 
namely $\alpha_{\eta I_0|\eta J^s}$, at $\overline{\pi}_{I_0|J^s}$.

Since the leading terms $\overline{\pi}_{I_0|J^s}$ of the vectors $\overline{v}(T^+_s)$ are linearly ordered with respect to $<$, and all other terms in 
$\overline{v}(T^+_s)$ are lower than $\overline{\pi}_{I_0|J^s}$,
if we express $\psi_k(\sigma\pi_{\eta I_0|\eta J^s})$ as a linear combination of elements $\overline{v}(T^+)$ for 
$T^+\in M({\lambda^+}'/\mu',\nu/\omega)$, then the coefficient at $\overline{v}(T^+_s)$ is $\alpha_{\eta I_0|\eta J^s}$ and 
all coefficients at $\overline{v}(T^+_t)$ for $t>s$ vanish.
\end{proof}

Before we proceed further, we need to adjust the sequence from the proof of  Proposition \ref{rob-link2} to make it adhere to our reading of the tableau $T^+_s$.
Corresponding to this reading, we define $\kappa^s_t=\lambda_{i_1\cdots i_t|j^s_1\cdots j^s_t}$ for $t=1, \ldots, k$ and $\kappa_0=\lambda$.

\begin{teo}\label{link}
Let $\lambda$ and $\lambda_{I|J}$ be dominant polynomial weights. 
If the simple $G$-supermodule $L_G(\lambda_{I|J})$ is a composition factor of $\nabla(\lambda)$, then $\lambda \sim_{podd} \lambda_{I|J}$. 
\end{teo}
\begin{proof}
We use Corollary 7.2 of \cite{fm-p2} and Proposition \ref{comp} and assume that the map $\psi_k:S_{I|J} \to \overline{S}_{I|J}$ is not surjective. 

Consider the collections $\mathfrak{C}_s=\mathfrak{C}_{T_s^+}$, consisting of vectors $\sigma\pi_{\eta I_0|\eta J^s}$ for $\eta \in \Sigma_k$, where $J^s$ corresponds to $T^+_s$, listed with respect to the order $T^+_1< \ldots <T^+_l$.

To each set $\mathfrak{C}_s$, we assign a set $O_s=\{\alpha_{\eta I_0|\eta J^s}|\eta\in \Sigma_k\}$.
It is a crucial observation that if every $O_s$, for $s=1, \ldots, l$, contains a nonzero element, then the map $\psi_k: S_{I|J} \to \overline{S}_{I|J}$ is surjective. 
This is because by Lemma \ref{surj4} and Theorem 7.1 of \cite{fm-p2} we can find a set of vectors $\sigma \pi_{\eta_s I_0|\eta_s J^s}$ of weight $\lambda_{I|J}$ such that the matrix of $\psi_k$ restricted on the span of these vectors is invertible.

Therefore, if $\psi_k: S_{I|J} \to \overline{S}_{I|J}$ is not surjective, then there is an index $s$ such that all elements of $O_s$ vanish. 

We show that 
\[\lambda=\kappa^s_0\sim_{sodd}\kappa^s_1\sim_{sodd}\cdots \sim_{sodd} \kappa^s_k=\lambda_{I|J}.\] 

Assume that $\eta$ is such that $\eta(k)=t$.
Because of the order of the reading of the tableau $T^+_s$ and because the tableau $T^+_s$ is semi-standard, we infer that $a_{J^s,\eta}$ equals the number 
of entries in $\mathcal{D}^+$ that appear in the $i_t$-th column that lie in rows with indexes less than $k^s_t$. 
Therefore, $(\kappa^s_{t-1})^+_{i_t}=\lambda^+_{i_t} -a_{J^s,\eta}$.

We claim that $b_{J^s,\eta}$ equals the number of appearances of the symbol $m+j^s_t$ in the initial part of the reading of $T^+_s$ consisting of the first $t-1$ elements.
To see this, observe that since $T^+_s$ is semistandard if the symbol $m+j^s_t$ appear in columns with index bigger than $i_t$, then it must lie in the rows with indexes not exceeding $k^s_t$. This means that these appearances of the symbol $m+j^s_t$ all lie in the initial part of the reading of $T^+_s$ consisting of the first $t-1$ elements. On the other hand, since $T^+_s$ is semistandard, there cannot be any appearances of $m+j^s_t$ that lie in rows with an index smaller than $k^s_t$ and in columns with indices smaller or equal to $i_t$.
This means that every appearance of $m+j^s_t$ in the initial part of the reading of $T^+_s$ consisting of the first $t-1$ elements also counts towards $b_{J^s,\eta}$.
Therefore, $(\kappa^s_{t-1})^-_{j^s_t}=\lambda^-_{j^s_t}+b_{J^s,\eta}$. This implies $\omega_{i_t,j^s_t}(\kappa^s_{t-1})=\alpha_{\eta I_0|\eta J^s}=0$ and $\kappa^s_{t-1}\sim_{sodd}\kappa^s_t$ for each $t=1, \ldots, k$, which proves $\lambda\sim_{odd} \lambda_{I|J}$. Since all $\kappa^s_t$ are polynomial weights, the claim follows.
\end{proof}

\section{Remarks for ground fields of odd characteristic}\label{6}

In this section, we assume that the characteristic of the ground field $F$ is $p>2$.

\begin{pr}
Let $M$ be a submodule of $H^0_G(\lambda)$ generated by all elements from $F_i$ for $i<k$. Let $\psi_k:T_k\to F_k$ and $\tilde{\psi}_k:F_{k-1}\otimes Y \to F_k$ 
be the maps as before. Then $M\cap F_k=\psi_k(T_k)=\tilde{\psi_k}(F_{k-1}\otimes Y)$.

For a dominant weight $\mu$ corresponding to an element of $F_k$, the supermodule $L_G(\mu)$ is a composition factor of $H^0_G(\lambda)$ if and only if the module $L_{ev}(\mu)$ is a composition factor of $F_k/M\cap F_k$.
\end{pr}
\begin{proof}
The equality $M\cap F_k=\psi_k(T_k)$ follows by similar arguments as in the proof of Proposition \ref{comp}. Moreover,
simple $G_{ev}$-composition factors $L_{ev}(\mu)$ of $F_k/M\cap F_k$ are in one-to-one correspondence to simple composition factors $L(\mu)$ of $H^0_G(\lambda)$ 
generated by elements from $F_k$.
\end{proof}

Since $Im(\psi_k)=Im(\tilde{\psi}_k)$, we can study the map $\tilde{\psi}_k$ instead of the map $\psi_k$.  
Note that whether $L(\mu)$ is a composition factor of $H^0_G(\lambda)$ depends only on the $F_{k-1}$ and on the map $\tilde{\psi}_k$, but not on other preceding floors $F_i$ for $i<k$.
If the $G_{ev}$-structure of $F_{k-1}\otimes Y$ and $F_k$ is known, we can investigate $\tilde{\psi}_k$ as a $G_{ev}$-morphism. We will not, however, pursue this direction in this paper.

We would like to discuss the modular reduction from the ground field of rational numbers $\mathbb{Q}$ to a ground field $F$ of characteristic $p>2$. 
From now on, assume that $\lambda$ is a dominant polynomial weight and $\mu=\lambda_{I|J}$ is a dominant polynomial weight belonging to the $k$th floor $F_k$ of $H^0_G(\lambda)$.
Denote by $S_{\mu,F}$ and $\overline{S}_{\mu,F}$ the sets of even-primitive vectors of weight $\mu$ in $T_k$ and $F_k$, defined over the field $F$.

Recall the definition of the sets $\mathfrak{B}$ and $\mathfrak{B}_{\eta}$ for $\eta\in\Sigma_k$ and their $F$-spans 
$B_{\eta, F}$ and $B_F$
from the beginning of Section \ref{sec5}.

Denote the $\mathbb{Z}$-span of elements from $\mathfrak{B}_{\eta}$ by $Z_{\eta,\mathbb{Z}}$ and 
the $\mathbb{Z}$-span of elements from $\mathfrak{B}$ by $Z_{\mathbb{Z}}$. 
Then $Z_{\eta,\mathbb{Z}}\otimes_{\mathbb{Z}} \mathbb{Q} \simeq B_{\eta,\mathbb{Q}}$, 
and $Z_{\mathbb{Z}}\otimes_{\mathbb{Z}} \mathbb{Q} \simeq B_{\mathbb{Q}}$.

Denote by $Z_{\eta}$ the image of $Z_{\eta,\mathbb{Z}}$ under the reduction modulo $p$, and by $Z$ the image of $Z_{\mathbb{Z}}$ under the reduction modulo $p$. 
Over a ground field $F$ of positive characteristic $p>2$, the space $B_{\eta,F}$ contains $Z_{\eta,\mathbb{Z}}\otimes_{\mathbb{Z}} \mathbb{F}$ but is bigger in general, 
and the space $B_{F}$ contains $Z_{\mathbb{Z}}\otimes_{\mathbb{Z}} \mathbb{F}$ but is bigger in general.

Let us modify the definition of the simple-odd-linkage of weights in the case when the ground field $F$ has characteristic $p>2$ by replacing the requirement 
$\omega_{ij}=0$ by $\omega_{ij}\equiv 0 \pmod p$.

It follows from Lemma \ref{surj4} that $\psi_{k, \mathbb{Q}}(Z_{\eta, \mathbb{Z}})\subseteq Z_{\mathbb{Z}}$. Moreover, 
$\psi_{k, \mathbb{Q}}: B_{\eta,\mathbb{Q}} \to B_{\mathbb{Q}}$ is induced by $\psi_{k, \mathbb{Z}}: Z_{\eta, \mathbb{Z}}\to Z_{\mathbb{Z}}$.
When we reduce the map $\psi_{k,\mathbb{Z}}:Z_{\eta,\mathbb{Z}} \to Z_{\mathbb{Z}}$ modulo $p$, we obtain a map $\psi^Z_{\eta}: Z_{\eta} \to Z$, 
which is a restriction and corestriction of the map $\psi_{k,F}:B_{\eta,F} \to B_{F}$.
Combine different maps $\psi^Z_{\eta}$ to a map $\psi^Z:\oplus_{\eta\in \Sigma_k} Z_{\eta} \to Z$ which is a restriction and corestriction of 
$\psi_k:S_{\mu,F} \to \overline{S}_{\mu,F}$.

The next statement gives a connection to the linkage principle for general linear supergroups over the field of characteristic $p>2$ - see \cite{mz}.

\begin{pr}
Assume $\lambda$ and $\mu=\lambda_{I|J}$ are dominant and polynomial weights, and the characteristic of the ground field $F$ is $p>2$.
If $\psi^Z$ is not surjective, then $\lambda\sim_{odd} \mu$.
\end{pr}
\begin{proof}
We proceed as in the proof of Theorem \ref{link} and  
consider the collections $\mathfrak{C}_s=\mathfrak{C}_{T_s^+}$, consisting of vectors $\sigma\pi_{\eta I_0|\eta J^s}$ for $\eta \in \Sigma_k$, where $J^s$ corresponds to $T^+_s$, listed with respect to the order $T^+_1< \ldots <T^+_l$.

To each set $\mathfrak{C}_s$, we assign a set $O_s=\{\alpha_{\eta I_0|\eta J^s}|\eta\in \Sigma_k\}$.
It is a crucial observation that if every $O_s$, for $s=1, \ldots, l$, contains a nonzero element, then the map $\psi^Z$ is surjective. 
This is because by Lemma \ref{surj4} and Theorem 7.1 of \cite{fm-p2} we can find a set of vectors $\sigma \pi_{\eta_s I_0|\eta_s J^s}$ of weight $\lambda_{I|J}$ such that the matrix of $\psi_k$ restricted on the span of these vectors is invertible.

Therefore, if $\psi_Z$ is not surjective, then there is an index $s$ such that all elements of $O_s$ vanish. 
The remainder of the proof is analogous to the second half of the proof of Theorem \ref{link}.
\end{proof}
\section*{Acknowledgment.}
The author thanks an anonymous referee for careful reading of the manuscript and for suggesting improvements that increased its readability.

\end{document}